\documentclass[a4paper,10pt]{article}
\usepackage[plainpages=false]{hyperref}
\usepackage{amsfonts,latexsym,rawfonts,amsmath,amssymb,amsthm, mathrsfs, lscape, calligra}
\usepackage{verbatim}
\usepackage{enumerate}
\usepackage{centernot}
\usepackage{mathtools}
\usepackage{tikz}
\usepackage{authblk}

\usepackage[all]{xy}
\usepackage{graphicx,psfrag}

\usepackage{array, tabularx, color}

\usepackage{setspace}

\newtheorem{thm}{Theorem}[section]
\newtheorem{cor}[thm]{Corollary}
\newtheorem{lem}[thm]{Lemma}
\newtheorem{clm}[thm]{Claim}
\newtheorem{prop}[thm]{Proposition}

\newtheorem{conj}[thm]{Conjecture}

\theoremstyle{remark}
\newtheorem{rmk}[thm]{Remark}

\theoremstyle{definition}
\newtheorem{defi}[thm]{Definition}

\def \E {\mathcal{E}}
\def \F {\mathcal{F}}
\def \O {\mathcal{O}}
\def \T {\mathcal{T}}

\def \bp {\bar{\partial}}

\def \C {\mathbb C}
\def \Z {\mathbb Z}

\def \F {\mathcal F}
\def \E {\mathcal E}
\def \V {\mathcal V}

\def \P {\mathbb P}

\def \T {\mathcal T}

\def \p {\partial}
\def \bp {\bar{\partial}}

\def \dVol {\text{dVol}}

\def \O {\mathcal{O}}
\def \l0 {\lim_{r\rightarrow 0}}

\def \Id {\text{Id}}
\def \Gr {Gr}
\def \Tr {\text{Tr}}
\def \Sing {\text{Sing}}
\def \S {\mathcal S}

\def \I {\mathcal I}
\DeclareMathOperator{\tr}{tr}
\DeclareMathOperator{\rk}{rank}

\def \C {\mathbb C}
\def \Z {\mathbb Z}

\def \F {\mathcal F}

\def \P {\mathbb P}

\def \p {\partial}
\def \bp {\bar{\partial}}

\def \dvol {\text{dVol}}

\def \O {\mathcal{O}}
\def \l0 {\lim_{r\rightarrow 0}}

\def \E {\mathcal{E}}
\def \Tr {\text{Tr}}

\numberwithin{equation}{section}
\begin{document}

\title{Analytic tangent cones of admissible Hermitian-Yang-Mills connections}
\date{\today}
\author{Xuemiao Chen\thanks{University of Maryland, xmchen@umd.edu}, Song Sun\thanks{UC Berkeley, sosun@berkeley.edu. }}

\maketitle
\begin{abstract}
In this paper we study the analytic tangent cones of admissible Hermitian-Yang-Mills connections near a homogeneous singularity of a reflexive sheaf, and relate it to the  Harder-Narasimhan-Seshadri filtration. We also give an algebro-geometric characterization of the bubbling set. 
\end{abstract}
\tableofcontents

\section{Introduction}
This article is a continuation of \cite{CS1} on studying tangent cones of (isolated) singularities of admissible Hermitian-Yang-Mills connections. The goals are the following
\begin{itemize}
\item Remove a technical assumption in the main theorem in \cite{CS1} by using a \emph{different} argument;
\item Study a \emph{stronger} notion of analytic tangent cones by including the information on the analytic bubbling sets;
\item Give an algebro-geometric characterization of the bubbling sets.
\end{itemize} 

Now we recall the main set-up, following \cite{CS1}. Let $B=\{|z|<1\}\subset \C^n$ be the unit ball endowed with a smooth K\"ahler metric $\omega=\omega_0+O(|z|^2)$ (where $\omega_0=\sqrt{-1}\p\bp |z|^2$ is the standard flat metric) and let $A$ be an admissible Hermitian-Yang-Mills connection on $B$.  Then $A$ defines a reflexive sheaf $\mathcal E$ over $B$. In this paper we always assume $0$ is an isolated singular point of $A$. Our goal is to understand the infinitesimal structure of $A$ near $0$ in terms of the complex/algebraic geometric information on the stalk of $\mathcal E$ at $0$. Loosely speaking we are searching for an \emph{analytic/algebraic} correspondence, which can be viewed as a \emph{local} analogue of the well-known Donaldson-Uhlenbeck-Yau theorem. 

From the analytic point of view, we can take \emph{analytic tangent cones} of $A$ at $0$, which are defined as follows. Let $\lambda: z\mapsto \lambda z$ be the rescaling map on $\C^n$. Then by Uhlenbeck's compactness result (\cite{ Nakajima,Tian, UY}),  we know as $\lambda\rightarrow 0$, by passing to a subsequence and applying gauge transforms, the rescaled sequence of connections $A_\lambda:=\lambda^*A$ converge to a smooth Hermitian-Yang-Mills connection $A_\infty$ on $\C^n_*\setminus \Sigma$. Here $\C^n_*:=\C^n\setminus\{0\}$, and  $\Sigma$ is a closed subset of $\C^n_*$ that has locally finite Hausdorff codimension four measure, and  we may assume $\Sigma$ is exactly the set where the convergence is not smooth. We call $\Sigma$ the \emph{analytic bubbling set}\footnote{For our purpose in this paper we will always remove the point $0$ and we only consider the convergence of \emph{smooth} connections, locally away from $0$, so that we can directly use the Uhlenbeck convergence theory. In general one could try to understand the bubbling set of a sequence of \emph{admissible} Hermitian-Yang-Mills connections, which we leave for future study.}. By Bando-Siu \cite{BS}, $A_\infty$ extends to an admissible Hermitian-Yang-Mills connection on $(\C^n, \omega_0)$ and it defines a reflexive sheaf $\E_\infty$ on $\C^n$. By \cite{Tian} (see also the discussion in Section 2),  passing to a further subsequence we may assume the Yang-Mills energy of $A_\lambda$ weakly converges to a limit Radon measure $\mu$ on $\C^n$. Write
$$\mu=|F_{A_\infty}|^2\dVol_{\omega_0}+8\pi^2\nu, $$ 
and define the \emph{blow-up} locus as $\Sigma_b:=\text{Supp}(\nu)\setminus \{0\}$. We know that $\Sigma$ is always a complex-analytic subvariety of $\C^n_*$ and  $\Sigma_b$ consists of precisely the closure of the codimension two part of $\Sigma$, and to each irreducible component of $\Sigma_b$ one can associate an \emph{analytic multiplicity}; the lower dimensional strata corresponds to the essential singularities of the connection $A_\infty$ which can not be removed (see Theorem $4.3.3$ and Remark $5$ in \cite{Tian}). For more detailed discussion see Section 2. 

Throughout this paper, we shall call the triple $(A_\infty, \Sigma, \mu)$ an \emph{analytic tangent cone} of $A$ at $0$.  Compared to \cite{CS1}, the definition here includes the extra data of the bubbling set and the limit measure, hence contains more information. A priori $(A_\infty, \Sigma, \mu)$ depends on the choice of subsequences as $\lambda\rightarrow 0$. We also know that $A_\infty$ is a HYM cone connection in the sense of Definition $2.22$ in \cite{CS1} (see Theorem $2.25$ there). Namely,  the corresponding reflexive sheaf $\E_\infty$ on $\C^n$ is isomorphic to $\psi_*\pi^*\underline \E_\infty$, where $\pi: \C^n_*\rightarrow \C\P^{n-1}$ is the natural projection map and $\psi: \C^n_*\rightarrow \C^n$ is the inclusion map, and 
$$\underline \E_\infty=\bigoplus_j \underline \F_j$$
where each $\underline \F_j$ is a stable reflexive sheaf.  The connection $A_\infty$ is isomorphic to the direct sum of the pull-back of the (unique) Hermitian-Yang-Mills connection on each $\underline \F_j$ under the projection map $\pi$, modified by adding a term given by  $\mu_j$ times the pull-back of the Chern connection associated to the Fubini-Study metric on $\O(1)$ (this is necessary to  make the Einstein constant vanish). So in short the limit connection $A_\infty$ is uniquely characterized by the algebraic data $\underline \E_\infty:=\oplus_j\underline \F_j$. In the language of \cite{CS1}, each factor $\underline\F_j$ corresponds to a \emph{simple HYM cone}.  We emphasize again that the analytic tangent cone is a priori not known to be unique, since it depends on not only the connection $A$ but also the choice of subsequences.

From the complex-algebraic point of view,  in \cite{CS1} we introduced the notion of an \emph{algebraic tangent cone} at a singularity of a reflexive coherent analytic sheaf $\E$. This is defined to be a torsion-free sheaf on $\C\P^{n-1}$ that is given by the restriction of a reflexive extension of $p^*(\E|_{B\setminus\{0\}})$ across $p^{-1}(0)$, where $p: \hat B\rightarrow B$ is the blown-up at $0$. We point out that in general algebraic tangent cones are not necessarily unique either, due to the fact that the exceptional divisor has complex codimension exactly one.  

To make a connection between analytic and algebraic tangent cones,  we recall the following conjecture in \cite{CS1}. Let $\E$ be a reflexive sheaf over $B$ with isolated singularity at $0$.
\begin{conj}\label{Conj1.1}
\begin{enumerate}[(I).]
\item Given any admissible Hermitian-Yang-Mills connection $A$ on $\E$, all the analytic tangent cones  of $A$ at $0$ have gauge equivalent admissible connection $A_\infty$ and hence the underlying sheaf $\E_\infty$ is unique up to isomorphism; 
\item There is an algebraic tangent cone $\underline\E^{alg}$ on $\C\P^{n-1}$ such that for all admissible Hermitian-Yang-Mills connection $A$ on $\E$, the reflexive sheaf $\E_\infty$ corresponding to the analytic tangent cones is always isomorphic to $\psi_*\pi^*((Gr^{HNS}(\underline\E^{alg}))^{**})$, where $Gr^{HNS}$ means taking the graded object associated to the Harder-Narasimhan-Seshadri filtration. 
\end{enumerate}
\end{conj}

Notice (I) implies that $(A_\infty, \E_\infty)$ does not depend on the choice of subsequences when taking the limit $\lambda\rightarrow 0$, and (II) implies that it does not depend on the choice of the connection $A$ either  and is purely a complex algebraic geometric \emph{invariant} of the sheaf $\E$. 
One also would like to understand the construction and uniqueness of the algebraic tangent cones $\underline\E^{alg}$. These are sensible complex/algebro-geometric questions which will be studied in the future. 

In \cite{CS1} we studied the special case when $\E$ is isomorphic to $\psi_*\pi^*\underline \E$ for some locally free sheaf $\underline \E$ on $\C\P^{n-1}$ (in which case we call $0$ a homogeneous singularity of $\E$),   and we proved the above conjecture with the choice $\underline \E^{alg}=\underline \E$, but under a technical assumption that  $Gr^{HNS}(\underline \E)$ is reflexive. The first goal of this paper is to remove this technical restriction. Note that a posterior, $Gr^{HNS}(\underline \E)$ being reflexive is equivalent to $\Sigma_b$ being empty.  

\begin{thm}\label{main}
Suppose $\mathcal E$ is a reflexive sheaf on $B$ with $0$ as an isolated singularity, such that $\E$ is isomorphic to $(\psi_*\pi^* \underline\E)|_B$ for some holomorphic vector bundle $\underline\E$ over $\C\P^{n-1}$. Then for any admissible Hermitian-Yang-Mills connection $A$ on $\E$, all the tangent cones at $0$ have the connection $A_\infty$. More precisely,  the corresponding $\E_\infty$ is isomorphic to $\psi_*\pi^*(\Gr^{HNS}(\underline\E))^{**}$, and $A_\infty$ is gauge equivalent to the natural Hermitian-Yang-Mills cone connection that is induced by the admissible Hermitian-Yang-Mills connection on $(\Gr^{HNS}(\underline\E))^{**}$. Furthermore, $\pi^{-1}(\text{Sing}(Gr^{HNS}(\underline \E)))\subset \Sigma$ for any tangent cone $(A_\infty, \mu, \Sigma)$. 
\end{thm}

The main motivation for the generalization in Theorem \ref{main} is that we also want to understand the analytic bubbling set in terms of the given complex geometric data. Given a torsion free sheaf $\F$, we define its singular set  $\text{Sing}(\F)$ to be the set where $\F$ fails to be locally free. 

\begin{thm}\label{Bubbling set}
Under the same hypothesis as Theorem \ref{main}, the analytic bubbling set $\Sigma$ is also independent of the choice of subsequences. Moreover, it agrees with the singular set $\Sigma^{alg}$ of $\pi^*(Gr^{HNS}(\underline \E))$ as a set and for each irreducible codimension $2$ component, the analytic multiplicity agrees with the algebraic multiplicity. In particular, the limit measure $\mu$ is also uniquely determined by $\underline \E$.
\end{thm}
For the definition of algebraic multiplicity we refer to Section \ref{UniquenessOfBubblingSet}. Notice $\text{Sing}(\E_\infty)\setminus\{0\}$ is obviously a subset of $\Sigma^{alg}$, and by Theorem \ref{main} the difference only appears when $Gr^{HNS}(\underline\E)$ fails to be reflexive. One particular interesting fact is that there are examples where 
$Gr^{HNS}(\underline\E)$ is not reflexive and its double dual is a direct sum of line bundles, so $\psi_*\pi^*(Gr^{HNS}(\underline\E))^{**}$ is trivial, i.e. $\psi_*\pi^*(Gr^{HNS}(\underline\E))^{**}\cong \O_{\C^n}^{\oplus\rk(\underline \E)}$. 

\begin{cor}\label{Example}
There exists an admissible Hermitian-Yang-Mills connection on a rank two reflexive sheaf over $\C\P^3$, such that at all of its singular points the analytic tangent cones have trivial flat connections but non-empty bubbling sets.
\end{cor}

More generally, in view of Theorem \ref{Bubbling set}, we expect a strengthening of Conjecture \ref{Conj1.1}.

\begin{conj}
In part (II) of Conjecture \ref{Conj1.1} we define the algebraic bubbling set $\Sigma^{alg}$ to be the singular set of $\pi^*(Gr^{HNS}(\underline \E^{alg}))$. Then $\Sigma=\Sigma^{alg}$ for all tangent cones and for each irreducible codimension $2$ component, the analytic multiplicity and the algebraic multiplicity are equal. 
\end{conj}

We now explain the main ideas in the proof of Theorem \ref{main}.  
 In the proof in \cite{CS1}, the technical restriction on $Gr^{HNS}(\underline \E)$ being reflexive is already needed when $\underline \E$ is semistable. So in the following we shall focus on the case when $\underline\E$ is semistable and the unstable case imposes no essential extra difficulties.
 
   It is known  that given an analytic tangent cone $(A_\infty, \Sigma, \mu)$, both the singular set of $A_\infty$ and the bubbling set $\Sigma$ are $\C^*$ invariant (see Lemma $5.5.1$ in \cite{Tian} and Theorem $2.23$ in \cite{CS1}), so we can assume our tangent cone is (up to isomorphism) given by the rescaling along a subsequence of the fixed sequence $\{\lambda_j:=2^{-j}\}$. For simplicity we denote by $A_j$ the pull-back connection $\lambda_j^*A$ on the ball $B$. Notice by the nature of tangent cones it suffices to restrict our attention to the unit ball $B$. 

Recall in \cite{CS1} we view $A$ as the Chern connection of an admissible Hermitian-Einstein metric $H$ on the reflexive sheaf $\E$. For each holomorphic section $s$ of $\E$ we defined the notion of the \emph{degree}  $d(s)$, which is a number that measures the vanishing order of $s$ at $0$, with respect to the unknown metric $H$. The fact that $d(s)$ is well-defined depends on a key convexity property similar to the classical three circle lemma. Assuming $\underline \E$ is semistable we proved that the degree of all non-zero sections of the form $\pi^*\underline s$ with $\underline s\in H^0(\C\P^{n-1}, \underline \E)$ is all the same  and is given by an explicit formula in terms of  the slope  of $\underline\E$. Under the rescalings, any non-zero holomorphic section $\pi^*\underline s$, by passing to subsequences and by a suitable normalization, gives rise to holomorphic sections on any tangent cone $\E_\infty$,  which are homogeneous of the degree $d(s)$ with respect to the natural cone structure on $\E_\infty$. 

 Now let $0\subset\underline \E_1\subset\cdots \underline\E_m=\underline \E$ be a Seshadri filtration of $\underline\E$, and the goal is to build isomorphisms from each quotient $\psi_*\pi^*(\underline \E_p/\underline \E_{p-1})^{**}$ to a direct summand of $\E_\infty$. By tensoring with some $\O(k)$ we can always assume each $\underline\E_p/\underline \E_{p-1}$ is generated by global holomorphic sections of $\underline \E_p$, and we denote by $HG_p$ the sections of the $\E$ of the form $\psi_*\pi^*s$ with $s\in H^0(\C\P^{n-1}, \underline \E_p)$. Then under the rescaling map, sections in $HG_p$ give rise to sections of $\E_\infty$ and they can be used to build non-trivial maps from $\E_p$ to $\E_\infty$. Here for each $j$, we need to normalize the sections in $HG_p$ by a common factor depending on $j$, so that the limit map is well-defined. 
 
When $p=1$ using the stability of $\E_1$ we obtain a splitting $\E_\infty=\S_1\oplus \V_1$ which is orthogonal on the locally free part, such that sections of $HG_1$ yields an isomorphism between $\E_1$ and $\S_1$ which descends to an isomorphism between $\underline \E_1$ and $\underline \S_1$ on $\C\P^{n-1}$. Notice $\underline \E_1$ is always reflexive, see Remark $2.8$ in \cite{CS1}. 

When $p=2$, a complication arises since it could happen that the sections of $HG_2$, under normalization by a common factor, may limit to sections of $\S_1$ too. So this does not immediately give rise to a new direct summands of $\E_\infty$. The approach we take in \cite{CS1} is that for each $j$, we perform $L^2$ orthogonal projection of elements in $HG_2\setminus HG_1$ on the rescaled ball $B$, to the orthogonal complement of $HG_1$.  Then we proved that the projected sections, after a common normalization, still give rise to holomorphic sections in $\E_\infty$. These generate a homomorphism $\psi: \E_2/\E_1\rightarrow \V_1$. 

By construction these limit sections are $L^2$ orthogonal to the sections in $\E_\infty$ that arise as limits of $HG_1$, hence are $L^2$ orthogonal to sections of $\S_1$ of the same homogeneous degree. From this we conclude $\psi$ is non-trivial by using $\underline \E_1$ being reflexive. Then using the stability of $\E_2/\E_1$ we get a splitting 
 $$\V_1=\S_2\oplus \V_2 $$
 and we get an map $\E_2/\E_1\rightarrow \S_2$ that induces an isomorphism $(\E_2/\E_1)^{**}\simeq \S_2$. 
 
Now we can try to continue this process, but we meet serious issues when $p\geq 3$. One can still construct the map from $\E_3/\E_2\rightarrow \V_2$ by the $L^2$ projection technique. However it is no longer easy to see that this map is \emph{non-trivial}. The reason is that by construction we only know the limits of the projected sections are $L^2$ orthogonal to the limit sections of $HG_2$, but if $\E_2/\E_1$ is not reflexive these latter sections do not necessarily span all the sections of $\S_2$ of the same homogeneous degree. There are possible ways to get around this difficulty when $p=3$. But we find this argument become tedious and very complicated when $p$ becomes larger. 
 
Instead in this paper we use a new idea to overcome this issue, and the key point is to replace $L^2$ projection by a \emph{pointwise} orthogonal projection. This will help overcome the above issue but in the meantime create new technical points that we now discuss. If we assume all the tangent cones consist of smooth connections without  bubbling set so that the rescaled connections converge smoothly, then it is relatively easy to see that the pointwise orthogonal projection still possesses convexity (in the form of a three circle type lemma) so that one can almost repeat the proof in \cite{CS1}. However this assumption can not be guaranteed  a priori and a posteriori by our main results it must not be satisfied if one of the factor $\E_p/\E_{p-1}$ is not locally free. Consequently in general we can only perform pointwise orthogonal projection away from the union of the singular sets of $\E_p/\E_{p-1}$, and in order this orthogonal projection behaves well as $j\rightarrow\infty$ we need to work on the complement of the analytic bubbling set $\Sigma$. 
 
 Now for simplicity of discussion we first assume $\Sigma$ is independent of the choice of analytic tangent cones, then we can simply cut off a fixed small neighborhood of $\Sigma$, and do pointwise orthogonal projection on the complement, say $\Omega$. Then we meet a common issue as in many problems in geometric analysis, namely, how do we take non-trivial limits of the projected sections as $j\rightarrow\infty$. If we normalize any reasonable norm to be 1, then general elliptic theory only guarantees interior estimates, and we can not exclude the possibility that the limit is zero, unless we can derive the estimates near the boundary of $\Omega$. Such an estimate can not follow from general elliptic theory, and it is at this point we rely crucially on the  complex geometry: the fact that the bubbling set $\Sigma$ is a complex-analytic subvariety of codimension at least 2 allows us to get uniform estimates on $\Omega$ for a fixed subsequence. Roughly speaking, one can choose $\Omega$ and a relatively compact $\Omega'\subset \Omega$ so that every point in $\Omega\setminus \Omega'$ lies on a holomorphic disc $D$ which is contained in the complement of $\Sigma$ and with boundary $\p D$ contained in $\Omega'$. Then we can restrict a holomorphic section $s$ to $D$ and use maximum principle on $D$ to conclude that the $L^\infty$ norm of $s$ over $\Omega$ can be uniformly controlled by the $L^\infty$ norm of $s$ over $\Omega'$, which can be controlled by the $L^2$ norm of $s$ over $\Omega$. This improved estimate on $s$ allows us to adapt most techniques in \cite{CS1} to the new setting to prove a key convexity result (Proposition \ref{prop2.18}). Using this and the Hartogs extension property of holomorphic sections, we are able to obtain limit holomorphic sections which, away from $\Sigma$, are pointwisely orthogonal to the sub-bundle of $\E_\infty$ obtained in the previous induction step. This then fixes the issue in the above discussion. 
 
 In general we do not know a priori that the bubbling set $\Sigma$ is independent of the choice of analytic tangent cones, and a priori the union of the bubbling sets of all the tangent cones could be the whole $\C^n_*$,  so we can not a priori cut off the region in terms of neighborhood of bubbling sets. Instead for each $j$ we shall cut off an intrinsic \emph{high curvature} region that depends on $j$, which as $j$ tends to infinity should be close to the neighborhoods of bubbling sets. This is a very delicate point and we refer to Section 2 for details. In Section 3 we shall prove Theorem \ref{main} following the above line of discussion.  
   
 The proof of Theorem \ref{Bubbling set} is essentially a direct consequence of Theorem \ref{main}, together with the formula on representing the analytic multiplicity in terms of curvature concentration (see Lemma $4.1$ in \cite{SW}), and a formula on representing the algebraic multiplicity in terms of a Chern-Simons transgression form (see Equation $(4.5)$ in \cite{SW}). This has been used to identify the analytic multiplicties and the algebraic multiplicities for the blow-up locus in the Hermitian-Yang-Mills flow case (see \cite{SW}). For the convenience of readers we will make a self-contained discussion in our setting. 

\

\noindent \textbf{Acknowledgements:} Both authors are supported by  the Simons Collaboration Grant on Special Holonomy in Geometry, Analysis, and Physics (488633, S.S.).   S. S. is partially supported by an Alfred P. Sloan fellowship and NSF grant DMS-1708420.  This paper forms a 
part of X.M. Chen's PHD thesis in Stony Brook University. He would also like to thank UC Berkeley for hospitality during his visit between January 2018 and May 2019, when this paper was written. Both authors are thankful to the anonymous referees for numerous suggestions which greatly improved the exposition of the paper.

\

\section{A convexity result}
In this section, let $(\E, A)$ be an admissible Hermitian-Yang-Mills connection over $(B=\{|z|<1\}\subset \C^n, \omega)$ with an isolated singularity at $0$.  As in \cite{CS1} we will often omit the volume form in an integral, and it is understood that we use the natural volume form underlying the connection.  In the following, the closure of a set is always taken in $\mathbb{C}^n_*$.

\subsection{Analytic tangent Cones}\label{Tangent Cone}
We first recall known results (c.f. \cite{ Nakajima, Price, Tian, UY}) on the convergence of a sequence of Hermitian-Yang-Mills connections with locally uniformly bounded Yang-Mills energy, adapted to our setting of getting analytic tangents cones. 

As in the introduction, for any $\lambda \in (0, 1]$ we consider the rescaling map defined by 
$$\lambda: B_{\lambda^{-1}}\rightarrow B; z\mapsto\lambda z$$ and denote 
$$A_{\lambda}:=\lambda^* A.$$
Then $A_{\lambda}$ is Hermitian-Yang-Mills with respect to the metric $\omega_\lambda:=\lambda^{-2} \cdot \lambda^*\omega$. 
 Given any subsequence $\lambda_i\rightarrow 0$, by Price's monotonicity formula \cite{Price} (see also Page 20, Remark 3 in \cite{Tian}), for any $R>0$,  the sequence $\{A_{\lambda_i}\}_i$ has uniformly bounded Yang-Mills energy over ${B_R}\setminus\{0\}$. Then by Uhlenbeck's compactness result (\cite{Nakajima, Tian, UY}) after passing to a subsequence, we may assume $\{A_{\lambda_i}\}_i$ converges locally smoothly to $A_\infty$ on $\C^n_*\setminus \Sigma$ modulo gauge transformations, where $\Sigma$ is a closed subset of $\C^n_*$ so that the Hausdorff $(2n-4)$ measure of $\Sigma\cap B_R$ is finite for any fixed $R>0$. More explicitly, we have
\begin{equation}\label{eqnBubblingSet}
\Sigma= \{z\in \C^n_*| \lim_{r\rightarrow0}\liminf_{i\rightarrow\infty} r^{4-2n}\int_{B_r(z)}|F_{A_{\lambda_{i}}}|^2 \geq \epsilon_0\}
\end{equation}
where $\epsilon_0>0$ denotes the constant in the $\epsilon$-regularity theorem (see Equation $(3.1.4)$ in \cite{Tian}). We denote by $\text{Sing}(A_\infty)$  the set of essential singularities of $A_\infty$ on $\C^n_*$ i.e. where $A_\infty$ can not be extended smoothly after a gauge transform on $\C^n_*$. Clearly $\Sing(A_\infty)\subset\Sigma$, but in general $\text{Sing}(A_\infty)$ may be strictly smaller due to the removable singularities of $A_\infty$. Passing to a further subsequence,  we may assume that the sequence of Radon measures $\{\mu_i:=|F_{A_{\lambda_i}}|^2\dvol_{\omega_{\lambda_i}}\}_i$ converge weakly to $\mu$ on $\C^n$. We define the triple $(A_\infty, \Sigma, \mu)$ to be an \emph{analytic tangent cone} of $A$ (associated to the chosen subsequence), and $\Sigma$ is called the \emph{analytic bubbling set}. For simplicity of notation, we denote
 $$
\lim_{i\rightarrow \infty} A_{\lambda_i} = (A_\infty, \Sigma, \mu). 
$$
By Fatou's lemma, there exists an non-negative measure $\nu$ on $\C^n$ so that
$$
\mu=|F_{A_\infty}|^2 \dVol_{\omega_0}+8\pi^2\nu.
$$
By Equation $(3.1.11)$ in \cite{Tian}, $\text{supp}(\nu)\setminus \{0\}$ is the \emph{blow-up locus} $\Sigma_b$ of the sequence $\{A_{\lambda_i}\}_i$ given as 
$$\Sigma_b=\overline{\{x\in \C^n_*|\Theta(\mu, x)>0,\lim_{r\rightarrow 0} r^{4-2n}\int_{B_r(x)}|F_{A_\infty}|^2=0 \}}.$$
where $\Theta(\mu, x):=\lim_{r\rightarrow0}r^{4-2n}\mu(B_r(x))$ is called the density function. It is easy to see that
\begin{equation}
\Sigma=\Sigma_b\cup \text{Sing}(A_\infty). 
\end{equation}

The removable singularity theorem in \cite{BS} implies that $A_\infty$ defines a reflexive sheaf $\E_\infty$ on $\C^n$, and we have
$$\Sing(A_\infty)=\Sing(\E_\infty)\setminus\{0\}.$$
In particular $\Sing(A_\infty)$ is a complex-analytic subvariety of $\C^n_*$. As a consequence of the monotonicity formula, Tian (\cite{Tian}, Lemma 5.3.1) proved that the connection $A_\infty$ is radially invariant, so is its singular set $\Sing(A_\infty)$. Therefore $\Sing(A_\infty)$ is $\C^*$ invariant, which implies $\pi(\Sing(A_\infty))$ is an algebraic subvariety of $\C\P^{n-1}$ (see Theorem 2.23 in \cite{CS1}). Also the invariance of $A_\infty$ implies that  for any $r\in (0, 1)$,  the function $$z\mapsto (|z|r)^{4-2n}\int_{B_{|z|r}(z)} |F_{A_\infty}|^2$$ is invariant under the natural $\C^*$ action on $\C^n_*$. This can be easily seen from the elementary fact that $|F_{A_\infty}|^2(t z)= |t|^{-4}|F_{A_\infty}|^2(z)$ for any $t \in \C^*$. By Theorem 4.3.3 in \cite{Tian}, we know\footnote{The proof in \cite{Tian} is written in the case of compact manifolds, but as remarked in \cite{Tian} (Remark 5), one only requires the boundedness of local Yang-Mills energy, which is valid in our case due to Price's monotonicty formula. 
} that $\Sigma_b$ is also a complex-analytic subvariety of $\C^n_*$ of pure codimension two (see also Lemma $3.2.3$ in \cite{Tian}), with finitely many irreducible components $\Sigma_k$, and there are positive integers $m_k$ such that the following current equation holds on $\C^n_*$
\begin{equation}\label{eqncurrent}
\lim_{i\rightarrow \infty} \frac{1}{8\pi^2}\tr(F_{A_{\lambda_i}}\wedge F_{A_{\lambda_i}})=\frac{1}{8\pi^2}\tr(F_{A_\infty} \wedge F_{A_\infty})+\sum m_k^{an}[\Sigma_k].
\end{equation}
In particular, 
$$\nu=\sum_k  m_k^{an}\mathcal{H}^{2n-4}|_{\Sigma_k}$$
where $\mathcal{H}^{2n-4}|_{\Sigma_k}$ denotes the $(2n-4)$ dimensional Hausdorff measure on $\Sigma_k$. In the following, abusing notation, we consider $\Sigma_b$ as a set with multiplicities $\Sigma_b=\sum_k m_k^{an} \Sigma_k$ (This will only be used in Section \ref{UniquenessOfBubblingSet}). Again by Lemma 5.3.1 in \cite{Tian}, we know $\Sigma_b$ is also radially invariant, hence it is also invariant under $\C^*$ action. 

Summarizing the above we have
\begin{lem} \label{lem3.3}
$\Sigma=\pi^{-1}(\underline \Sigma)$ where $\underline \Sigma$ is a subvariety of $\C\P^{n-1}$ of complex codimension at least $2$.
\end{lem}

Now fix a smooth point $z\in \Sigma_k$, and let $\Delta$ be a \emph{transverse slice} at $z$, i.e. $\Delta$ is a smooth complex two dimensional submanifold in $B$ such that $\Delta$ is transversal to $\Sigma_k$. The following is proved in \cite{SW} (see Lemma 4.1) and the argument is purely local. To make it more contained, we will roughly explain why it is true and for more details we refer the reader to \cite{SW}. 

\begin{lem}\label{AnalyticMultiplicity}
For  $\Delta$ which is a transverse slice at a generic point  $z\in \Sigma_k$, we have
\begin{equation}\label{AnalyticMultiplicityeq}
m^{an}_{k}=\lim_{i\rightarrow \infty} \frac{1}{8\pi^2} \int_{\Delta} \big\{\tr(F_{A_{\lambda_i}}\wedge F_{A_{\lambda_i}})-\tr(F_{A_{\infty}}\wedge F_{A_{\infty}})\big\}.
\end{equation}
\end{lem}

\begin{proof}
For any generic point $z\in \Sigma_k$ where $\Sigma_k$ is smooth, given any $\delta>0$, abusing notation, we use $\delta$ to denote the rescaling map of $B$ centered at $z$ and $V:=T_z\Sigma_k$. We denote $A_{\lambda_i, \delta}=\delta^* A_{\lambda_i}$ and $A_{\infty, \delta}=\delta^* A_{\infty}$. Then there exists a sequence $\{\delta_i\}$ so that $\delta_i \rightarrow 0$ and
$$
m_k^{an}=\frac{1}{8\pi^2}\int_{V^{\perp}\cap B_z(1)} Tr(F_{A_{\lambda_{j_i}, \delta_i}} \wedge F_{A_{\lambda_{j_i}, \delta_i}})- Tr(F_{A_{\infty, \delta_i}} \wedge F_{A_{\infty, \delta_i}})
$$
(see Equation (4.2.7) in \cite{Tian}). Here we identify the tangent space of $B$ at $z$ with $\C^n$ naturally and $V^{\perp}$ denotes the orthogonal complement of $V$ in $\C^n$. By doing integration by parts, it is shown in \cite{SW} that the term on the right hand side above differs from that in Equation \ref{AnalyticMultiplicityeq}
by the limit of some boundary term, which converge to zero. This finishes the proof. 
\end{proof}

\begin{rmk}\label{rmk2.3}
Lemma \ref{AnalyticMultiplicity} holds for any irreducible codimension $2$ subvariety  $\Sigma_k$ which is not necessarily a component of $\Sigma$. Indeed, $m_k^{an}=0$ in this case. 
\end{rmk}

The radial invariance of tangent cones has a few easy consequences, which will be used frequently later. 
\begin{cor}\label{cor2.5}
For $z\in B\setminus \{0\}$,
\begin{enumerate}[(a).]
\item $\lim_{i\rightarrow\infty}\mu_i(B_r(z))=\mu(B_{r}(z))$ for $r<|z|$;
\item $\lim_{s\rightarrow r} \mu(B_s(z))=\mu(B_{r}(z))$;
\item $\lim_{i\rightarrow\infty} \mu_i(B_{r_i}(z_i))=\mu(B_r(z))$, for  $z_i\rightarrow z, r_i\rightarrow r$.
\end{enumerate}
\end{cor}
\begin{proof}
For (a), by general theory on convergence of Radon measures it suffices to show that $\mu(\p B_r(z))=0$. Since $\Sigma=\pi^{-1}(\underline \Sigma)$ where $\underline \Sigma$ is a complex subvariety of real codimension $4$ in $\C\P^{n-1}$, $\Sigma\cap \p B_{r}(z)$ is of Hausdorff codimension at least $5$ in $\C^n$, hence we have $\mu(\p B_{r}(z))=0$. Now for $(b)$ we notice that $\Sigma$ being radially invariant implies that
$$|\mu(B_s(z))-\mu(B_{r}(z))|\leq  \Big| \int_{B_s(z)}|F_{A_\infty}|^2-\int_{B_r(z)}|F_{A_\infty}|^2\Big |+ C |s-r|$$
for some fixed constant $C$. So (b) follows. For $(c)$, fix $r<r'<|z|$ and for $i $ large one has $B_{r_i}(z_i) \subset B_{r'}(z)$. This implies 
$$\mu(B_{r'}(z))=\mu(\overline{B_{r'}(z)}) \geq \limsup_{i\rightarrow\infty}\mu_i(B_{r_i}(z_i)).
$$
By letting $r'\rightarrow r$, we have 
$$\mu(B_{r}(z)) \geq \limsup_{i\rightarrow\infty}\mu_i(B_{r_i}(z_i)).$$
Similarly one can prove  $\mu(B_{r}(z)) \leq \liminf_{i\rightarrow\infty} \mu_i(B_{r_i}(z_i))$. This finishes the proof.
\end{proof}

In our definition of analytic tangent cones we always need to pass to subsequences. For our later purpose we want to restrict to a particular discrete subsequence as $\lambda\rightarrow 0$. Namely, we define $\lambda_i:=2^{-i}$ and $A_i = \lambda_i^* A$. We say two analytic tangent cones are equivalent if they have the same bubbling set and the same analytic multiplicity of each irreducible Hausdorff codimension $4$ component and the corresponding connections are gauge equivalent. 
\begin{cor}
Any analytic tangent cone $(A_\infty, \Sigma, \mu)$ is equivalent to an analytic tangent cone arising from the limit of a subsequence of $\{A_i\}_i$. 
\end{cor}

For our purpose, we give more details about the convergence of a sequence of the rescaled connections and holomorphic sections when $\E=\pi^*\underline \E$. 
Suppose 
$$
\lim_i A_{j_i}=(A_\infty, \Sigma, \mu).
$$
By Theorem 2.23 in \cite{CS1}, we have
\begin{lem}\label{lem2.6}
 $A_\infty$ is a direct sum of simple HYM cones. Namely, after a gauge transform, $\E_\infty = \pi^{*}\underline \E_\infty$ and $ \underline \E_\infty$ is a direct sum of polystable reflexive sheaves on $\C\P^{n-1}$
$$\underline \E_\infty=\oplus_l \underline Q_l,$$
and
$$A_\infty=\bigoplus_{l} \pi^* \underline A_l+\mu_l\p \ln|z|^2\cdot\Id_{\pi^* \underline Q_l},$$
where $\underline A_l$ is the unique Hemitian-Yang-Mills connection on $\underline Q_l$ and $\mu_l$ denotes the slope of $\underline Q_l$. 
\end{lem}

For our purpose, we explain the meaning of  $\lim_i A_{j_i}=(A_\infty, \Sigma, \mu)$ in more details. Fix a smooth Hermitian $\underline H'$ on $\underline\E$, and let $H'=\pi^*\underline H'$.  Recall $H$ is the unknown Hermitian-Einstein metric on $\E$.   Let $H_i=(2^{-i})^*H$ and $f_i=(H'^{-1}H_i)^{\frac{1}{2}}$ be the complex gauge transform (note $f_i$ is Hermitian with respect to $H'$). Let $A_{i}$ be the Chern connection given by the hermitian metric $H$ and the holomorphic structure $f_i\cdot\bp_\E:=f_i\circ \bp_{\E}\circ f_i^{-1}$. Then there exists a unitary gauge isomorphism 
$$
P: (\E, H') \rightarrow (\E_\infty, H_\infty)
$$ 
outside $\Sigma$ and a sequence of unitary gauge transform $\{g_{j_i}\}_i$ of $(\E, H')$ defined outside $\Sigma$ so that $\{g_{j_i} \cdot A_{j_i}
\}_i$ converges to $P^{*} A_\infty$ smoothly outside $\Sigma$. 

Now given a sequence of holomorphic sections $\{\sigma_{i}\}$ of $\E$ over $B^*$, we know $f_i(\sigma_i)$ is a  holomorphic section of $(\E, f_i\cdot \bp_\E)$. 
\begin{defi}
$\{\sigma_{i}\}$ is called to converge to a holomorphic section $\sigma_\infty$ of $\E_\infty$, if $g_{j_i} f_{i}(\sigma_{j_i})$ converges locally smoothly to $P^{-1}\sigma_\infty$ away from $\Sigma$. 
\end{defi}
Since $g_{j_i}\cdot f_{i}(A_{j_i})$ converges to $P^*A_\infty$ outside $\Sigma$, by the elliptic regularity of $\bp$-operator,  we know that for any sequence of holomorphic sections $\{\sigma_i\}_i$ which are normalized suitably, by passing to subsequences, we can always obtain limit holomorphic sections of $\E_\infty$ in the above sense. However, the limit is \emph{not} a priori nontrivial. This is the reason why we develop the convexity results in the following subsections. 
\subsection{PDE estimate}
We will collect some PDE results from \cite{CS1} and refer the reader to Section $2.3$ in \cite{CS1} for detailed proof.  Let $\overline B^*=\overline B\setminus\{0\}\subset\C^n$. We always assume $n\geq 3$. For a function $g$, we denote $g^+=\max\{g,0\}$. 
In the following, the Laplacian $\Delta=\Delta_\omega$ is the analyst's Laplacian.  

\begin{lem}\label{pde}
Suppose $g\in C^{2}(B^{*})\cap C^0(\overline B^*)$ with $\int_{ B^*}|g^+|^{\frac{n}{n-1}}<\infty$ and $f$ is a non-negative function on $B^*$. If on $B^*$ we have
\begin{equation} \label{eqn2-16}
\Delta g(z) \geq -{|z|^{-2}}{f(z)}, 
\end{equation}
 then the following hold,
 \begin{enumerate}[(1).]
\item For all $z\in B^*$,  $$g(z)\leq |g|_{L^{\infty}(\p B)}+\int_{B^*} G(z,w) |w|^{-2} f(w)dw, $$
 where $G(z, w)$ is the (positive) Green's function for $-\Delta$ on $B$. The inequality is only meaningful when the right hand side is finite. 
\item For all $z\in B^*$, 
\begin{equation*}
g(z)\leq C_0(|g|_{L^\infty(\p B)}+\sup_{|w-z|\leq |z|/2}|f(w)|+(-\log |z|) \sup_{r\in (0, 1]} r^{1-2n}\int_{\p B_r}|f|), 
\end{equation*} 
where $C_0$ depends only on $n$.
In particular, if $|f|_{L^\infty(B^*)}<\infty$, then
$$\limsup_{z\rightarrow 0} \frac{g(z)}{-\log |z|}\leq C_0\limsup_{r\rightarrow0} r^{1-2n}\int_{\p B_r}|f|.$$
\end{enumerate}
\end{lem}

\begin{lem}\label{curvature}
Suppose $\E$ is a holomorphic vector bundle over $B^*$ and $H$ and $H'$ are two metrics on $\E$ over $B^*$, then we have  
\begin{enumerate}
\item $\Delta \log \Tr H^{-1} H' \geq -C(|\Lambda_{\omega} F_H|+|\Lambda_{\omega} F_{H'}|)$ for some constant $C=C(n)$;
\item If  $|F_H|+|F_{H'}|\in L^{1+\delta}(B^*)$ for some $\delta>0$, then
$$
(\log \Tr H^{-1} H')^+ +(\log \Tr (H')^{-1} H)^+ \in L^{\frac{n}{n-1}(1+\delta)}(B^*).
$$
\end{enumerate}
Here the norms of curvatures of $H$($H'$) are with respect to the natural Hermitian metric $H$ ($H'$) itself.
\end{lem}

\subsection{Cut out high curvature region}

We first introduce some terminologies. We say a subset $E$ of an open (or closed) annulus $A \subset \mathbb{C}^n$ is \emph{symmetric} if for any $z\in E$, we have
$$\C^*.z\cap A\subset E.$$  For any subset $E\subset\overline{B}_{2^{-1}}\setminus B_{2^{-2}}$, we define its \emph{symmetrization} to be the smallest symmetric subset that contains $E$, i.e., the set $\pi^{-1}(\pi(E))\cap (\overline{B}_{2^{-1}}\setminus B_{2^{-2}})$, where recall $\pi: \C^n_*\rightarrow \C\P^{n-1}$.  Also below when we discuss convergence of compact subsets of $\overline B$,  it will always be with respect to the Hausdorff distance on the space of all compact subsets of $\overline B$. 

For any $r\in (0, 10^{-3}]$ and integer $j\geq 1$, we define  
\begin{equation}\label{Ejr}
E^r_{j} =\{z\in \overline B_{2^{-1}}\setminus B_{2^{-2}}: (|z|r)^{4-2n}\int_{B_{|z|r}(z)} |F_{A_{j}}|^2 \geq \frac{\epsilon_0}{2}\}.
\end{equation}
Given a tangent cone $(A_\infty,\Sigma, \mu)$, we define a symmetric set
\begin{equation}
N^r(A_\infty,\Sigma, \mu):=\{z\in \overline{B}\setminus B_{2^{-3}}: (|z|r)^{4-2n}\mu(B_{|z|r}(z)) \geq \frac{\epsilon_0}{2}\}.
\end{equation}
From the definition of $\Sigma$ we see that for any $r>0$,
$$\Sigma\cap (\overline{B}\setminus B_{2^{-3}})\subset N^r(A_\infty,\Sigma, \mu).$$
  For notational convenience,  we will sometimes simply denote $N^r(A_\infty,\Sigma, \mu)$ by $N^r$ if the relevant tangent cone is clear from the context.  Given a subsequence $\{A_{j_i}\}_i$ converging to $(A_\infty,\Sigma, \mu)$, we denote
$$\Sigma^r_{j_i}:= 2 E^r_{j_i-1} \cup E^r_{j_i} \cup 2^{-1}E^r_{j_i+1}. $$
Now we are ready to state the main (technical) theorem of this section.
\begin{thm}\label{Cut-Off}
 There exists $r_0\in (0, 10^{-3})$ such that for any $r\in(0, r_0]$, and for any given tangent cone $(A_\infty, \Sigma, \mu)=\lim_{i\rightarrow\infty} A_{j_i}$ the following hold
\begin{enumerate}
\item[(I).] Suppose $V_1$ and $V_2$ are limits of $E^r_{j_i}$ and $E^{r}_{j_i+1}$ respectively, then 
$$m(V_1\setminus V_2)=m(V_2\setminus V_1)=0$$ 
where $m(\cdot)$ denotes the Lebesgue measure on $\C^n$;
\item[(II).] $N^\frac{r}{2}\subset \Sigma^r_{j_i} \subset N^{2r}$ for $i$ large. Moreover,
 $$d((B\setminus \overline{B}_{2^{-3}})\setminus N^{2r}, N^{\frac{r}{2}})>0,$$
 $$\liminf_{i}d((B\setminus \overline{B}_{2^{-3}})\setminus N^{2r}, \Sigma_{j_i}^r)>0,$$
  and
   $$d((B\setminus \overline{B}_{2^{-3}})\setminus N^{\frac{r}{2}}, \Sigma)>0, $$
here the distance $d$ is defined using the flat metric $\omega_0$.
\item[(III).] There exists a constant $C=C(r)>0$ so that for any $z\in \overline{(\overline {B}_{2^{-1}}\setminus B_{2^{-2}})\setminus N^{\frac{r}{2}}}$, there exists a flat holomorphic disk  $D_z \subset B_{\frac{3}{4}}\setminus B_{2^{-2}}$ such that $D_z\cap \Sigma=\emptyset$, $\p D_z \subset (B_{\frac{3}{4}}\setminus\overline{B}_{2^{-2}})\setminus N^{2r}$ and 
$$\min\{d(D_z,\Sigma), d(\partial D_z, \partial {B_{2^{-2}}})\}\geq C.$$ Here by ``flat holomorphic disk" we mean $D_z$ is a disk centered at $z$ which lies in some complex plane in $\C^n$ perpendicular to $\C \cdot z$ at $z$.
\end{enumerate}
\end{thm}

\begin{rmk}
This result is developed for the proof of convexity result in Section \ref{convexity}.
When we study rescaled limits of holomorphic sections, we will normalize the sections over $(B_{2^{-1}} \setminus B_{2^{-2}})\setminus E^r_{i}$ by its $L^2$ norm and study the convergence over this region.  (II) will guarantee  the interior smooth convergence of the sequence of normalized holomorphic sections since the region we consider will stay away from the bubbling set when taking limits. (III) will allow us to use the maximum principle and thus get uniform $L^\infty$ bound over $(B_{2^{-1}} \setminus B_{2^{-2}})\setminus E^r_{j_i}$. Together they provide a crucial convergence result of $L^2$ norm when taking limits, see the proof of Proposition \ref{prop2.18}. Roughly speaking, we need to prevent the concentration of the holomorphic section near the boundary when we take limits.  
\end{rmk}

\subsubsection{Proof of Theorem \ref{Cut-Off} (I).}
Given a tangent cone $(A_\infty,\Sigma, \mu)$, we consider the following function 
$$
f: B_{2^{-1}}\setminus \overline{B_{2^{-2}}}\times (0, 10^{-3}) \rightarrow \mathbb{R}_{+}, (z,r) \mapsto (|z|r)^{4-2n}\mu(B_{|z|r}(z)).
$$

\begin{lem}\label{lem2.8}

$f(z,r)=\frac{\epsilon_0}{2}$ if and only if $\overline{B_{|z|r}(z)}\cap \Sigma=\emptyset$ and 
$$(|z|r)^{4-2n}\int_{B_{|z|r}(z)} |F_{A_\infty}|^2= \frac{\epsilon_0}{2}.$$
\end{lem}

\begin{proof}
The \textit{if} part follows directly from the definition. For the \textit{only if} part, suppose $\lim_i A_{j_i}=(A_\infty, \Sigma, \mu)$. If $f(z,r)=\frac{\epsilon_0}{2}$, by Corollary \ref{cor2.5}, there exists $r'>r$ such that  
$$f(z,r')\leq \frac{3\epsilon_0}{4},$$
hence  for $i$ large
$$(|z|r')^{4-2n}\int_{B_{|z|r'}(z)}|F_{A_{j_i}}|^2 \leq \frac{5}{6}\epsilon_0.$$
By the choice of $\epsilon_0$, $\{A_{j_i}\}_{i}$ converge to $A_\infty$ smoothly over $B_{|z|\frac{r'+r}{2}}(z)$ and $B_{|z|\frac{r+r'}{2}}(z)\cap \Sigma=\emptyset$. As a result,  
$$(|z|r)^{4-2n}\mu(B_{|z|r}(z))=(|z|r)^{4-2n}\int_{B_{|z|r}(z)} |F_{A_\infty}|^2= \frac{\epsilon_0}{2},$$
and  $\overline{B_{|z|r}(z)}\cap \Sigma=\emptyset$. This finishes the proof. 
\end{proof}

\begin{rmk}\label{rmk2.13}
The conclusion also holds if we replace $\frac{\epsilon_0}{2}$ by any $c<\epsilon_0$.
\end{rmk}

\begin{lem}\label{lem2.9}
For any fixed $r\in (0, 10^{-3}]$,  the set 
$$\{z\in B_{2^{-1}}\setminus \overline{B_{2^{-2}}}: f(z, r)=\frac{\epsilon_0}{2}\}$$
 is a symmetric real analytic subvariety of $B_{2^{-1}}\setminus\overline{ B_{2^{-2}}}$, and it is a proper subset if $\Sigma$ is non-empty.  In particular, if $\Sigma\neq \emptyset$, then
$$m(\{z\in B_{2^{-1}}\setminus \overline{B_{2^{-2}}}: f(z, r)=\frac{\epsilon_0}{2}\})=0.$$
\end{lem}

\begin{proof}
Locally near any smooth point, under a holomorphic frame, the Hermitian-Einstein metric $h_\infty$  on $\E_\infty$ satisfies the following elliptic equation 
$$
P(h_\infty):=\sqrt{-1}\Lambda_{\omega_0} \bp (h^{-1}_{\infty}\p h_{\infty})=0.
$$ 
Since the coefficients of $P$ are real analytic in $z$, it follows from Theorem 41 on page 467 in \cite{BA}  that $h_\infty$ is also real analytic in $z$. Therefore, the function 
$$Q: \{z\in B\setminus \{0\}: B_{|z|r}(z)\cap\Sigma=\emptyset\}\rightarrow\mathbb R; z\mapsto (|z|r)^{4-2n}\int_{B_{|z|r}(z)}|F_{A_\infty}|^2$$
 is real analytic. 
  Now by Lemma \ref{lem2.8},  we know 
$$
\begin{aligned}
&\{z\in B_{2^{-1}}\setminus \overline{B}_{2^{-2}}: f(z,r)=\frac{\epsilon_0}{2}\}\\
=&\{z\in B_{2^{-1}}\setminus \overline{B}_{2^{-2}}:(|z|r)^{4-2n}\int_{B_{|z|r}(z)}|F_{A_\infty}|^2 =\frac{\epsilon_0}{2}, B_{|z|r}(z)\cap \Sigma=\emptyset\}.
\end{aligned}
$$
This easily implies $\{z\in B_{2^{-1}}\setminus \overline{B_{2^{-2}}}: f(z, r)=\frac{\epsilon_0}{2}\}$ is a symmetric real analytic subvariety of $B_{2^{-1}}\setminus \overline{B}_{2^{-2}}$.  The last statement follows from well-known facts about the zero set of a real analytic function (see for example \cite{BM}). 
\end{proof}

\begin{rmk}\label{rmk2.14}
It also follows from the proof that for any fixed $z$ such that $B_{|z|r}(z)\cap \Sigma=\emptyset$, the function 
$$s\mapsto(|z|s)^{4-2n}\int_{B_{|z|s}(z)}|F_{A_\infty}|^2$$ is real analytic in $(0, r)$.  Then given any constant $C$, the set
 $$\{s\in (0, r): (|z|s)^{4-2n}\int_{B_{|z|s}(z)}|F_{A_\infty}|^2=C\}$$
  is either equal to $(0,r)$ or consists of finitely many points.  
\end{rmk}

\begin{rmk}
Notice here we are working on the tangent cone, which is always with respect to the standard flat K\"ahler metric $\omega_0$ on $\C^n$, even if the original HYM connection is defined for an arbitrary smooth background K\"ahler metric. 
\end{rmk}

\begin{prop}\label{prop2.16}
There exists $r'_0\in (0, 10^{-3})$ such that for any $r\in (0, r_0')$ and any tangent cone $(A_\infty,\Sigma, \mu)$,
$$m(\{z\in B_{2^{-1}}\setminus \overline{B_{2^{-2}}}: f(z, r)=\frac{\epsilon_0}{2}\})=0.$$
\end{prop}

\begin{proof}
Otherwise, by Lemma \ref{lem2.9}, we can find a sequence $r_i\rightarrow 0$ and for each $r_i$ there exists a tangent cone $(A_\infty(i),\Sigma(i),\mu(i))$ with $\Sigma(i)=\emptyset$ and 
$$\{z\in B_{2^{-1}}\setminus \overline{B_{2^{-2}}}: f_i(z, r_i)=\frac{\epsilon_0}{2}\}=B_{2^{-1}}\setminus \overline{B_{2^{-2}}}.$$
 Taking limits,  we obtain $(A_\infty, \Sigma, \mu)$ with $B_{2^{-1}}\setminus \overline{B_{2^{-2}}}\subset \Sigma$, which is impossible. This is a contradiction.
\end{proof}

Now we finish the proof of (I) for all $r\in (0, r_0']$. 
\begin{proof}[Proof of (I)]
 We  claim
$$(V_1 \setminus V_2 )\cup (V_2\setminus V_1) \subset \{z\in \overline B_{2^{-1}}\setminus B_{2^{-2}}: f(z,r)=\frac{\epsilon_0}{2}\}.$$
Given this claim, by Proposition \ref{prop2.16}, we have $m(V_1\setminus V_2)=m(V_2\setminus V_1)=0$. We only prove the claim for $V_1\setminus V_2$ and the proof for $V_2\setminus V_1$ is the same. Given any $z\in V_1 \setminus V_2 $, we need to show $f(z,r)=\frac{\epsilon_0}{2}$. By passing to a subsqequence, there exists a sequence of points $z_i\in E^r_{j_i}\setminus E^r_{j_i+1}$ converging to $z$. By definition, for each $z_i$, there exists $y_i\in \overline{B}_{2^{-1}}\setminus B_{2^{-2}}$ with $\pi(z_i)=\pi(y_i)$ satisfying 
$$
(|y_i|r)^{4-2n}\int_{B_{|y_i|r}(y_i)} |F_{A_{j_i}}|^2 \geq \frac{\epsilon_0}{2} 
$$
but 
$$
(|\frac{y_i}{2}|r)^{4-2n}\int_{B_{|\frac{y_i}{2}|r}(\frac{y_i}{2})} |F_{A_{j_i}}|^2 <\frac{\epsilon_0}{2}.
$$
By passing to a subsequence, we can assume $\{y_i\}_i$ converge to $y\in \overline B_{2^{-1}}\setminus B_{2^{-2}}$ with $\pi(y)=\pi(z)$. By Corollary \ref{cor2.5}, we have 
$$
(|y|r)^{4-2n}\mu(B_{|y|r}(y))\geq \frac{\epsilon_0}{2}
$$
and
$$
(|\frac{y}{2}|r)^{4-2n}\mu(B_{|\frac{y}{2}|r}(\frac{y}{2}))\leq \frac{\epsilon_0}{2}.
$$
By Corollary \ref{cor2.5},  we have
$$(|z|r)^{4-2n}\mu(B_{|z|r}(z))=(|y|r)^{4-2n}\mu(B_{|y|r}(y))=(|\frac{y}{2}|r)^{4-2n}\mu(B_{|\frac{y}{2}|r}(\frac{y}{2}))=\frac{\epsilon_0}{2}.$$
This finishes the proof.
\end{proof}

\subsubsection{Proof of Theorem \ref{Cut-Off} (II).}
Again suppose we are given a tangent cone $(A_{\infty}, \Sigma, \mu)$. 
\begin{lem}\label{lem2.15}
Suppose for some $0<r_1<r_2<10^{-3}$ and $z\in B\setminus \{0\}$ we have
$$(|z|r_1)^{4-2n} \mu(B_{|z|r_1}(z))=(|z|r_2)^{4-2n} \mu(B_{|z|r_2}(z)), $$
then exactly one of the following holds:
\begin{itemize}
\item $(|z|r_2)^{4-2n} \mu(B_{|z|r_2}(z))\geq \epsilon_0$;
\item $(|z|r)^{4-2n} \mu(B_{|z|r}(z))\equiv 0$ for any $r\leq r_2$. 
\end{itemize}
\end{lem}
\begin{proof}
 Suppose $(|z|r_2)^{4-2n}\mu(B_{|z|r_2}(z))< \epsilon_0$, then by Remark \ref{rmk2.13} we have $\overline{B_{|z|r_2}(z)}\cap \Sigma=\emptyset$. So on  $B_{|z|r_2}(z)$, $A_\infty$ is smooth and  $\mu=|F_{A_\infty}|^2dVol$. 
By Price's monotonicity formula, under the above assumption,  the following function
$$(|z|s)^{4-2n}\int_{B_{|z|s}(z)} |F_{A_\infty}|^2$$
is constant on $[r_1, r_2]$. 
So by Remark \ref{rmk2.14}, it is constant on $(0, r_2]$. The conclusion follows by letting $s$ tend to zero.

\end{proof}

\begin{proof}[Proof of (II)]
Suppose $\{A_{j_i}\}_i$ converges to $(A_\infty, \Sigma, \mu)$. We first show the inclusion $N^{\frac{r}{2}}\subset \Sigma^{r}_{j_i}$ for $i$ large. Otherwise, by passing to a subsequence, there exists a sequence of points $z_{j_i} \in N^{\frac{r}{2}}\setminus \Sigma^r_{j_i}$ and $z_{j_i}$ converges to $z\in N^{\frac{r}{2}}$. In particular, 
$$(|z|\frac{r}{2})^{4-2n}\mu(B_{|z|\frac{r}{2}}(z))\geq \frac{\epsilon_0}{2}.$$
 Since $z_{j_i}\notin \Sigma^r_{j_i}$,  by Corollary \ref{cor2.5}, we must have  
 $$(|z|r)^{4-2n}\mu(B_{|z|r}(z))\leq \frac{\epsilon_0}{2}. $$
By Price's monotonicity formula (see Equation $(5.3.4)$ in \cite{Tian}), we have
 $$(|z|\frac{r}{2})^{4-2n}\mu(B_{|z|\frac{r}{2}}(z))=(|z|r)^{4-2n}\mu(B_{|z|r}(z))=\frac{\epsilon_0}{2}.$$ 
However, by Lemma \ref{lem2.15}, this is impossible. Similarly, one can show $\Sigma^r_{j_i} \subset N^{2r}$ for $i$ large and thus
$$\liminf_i d((B\setminus \overline{B_{2^{-3}}})\setminus N^{2r}, \Sigma^r_{j_i})>0.$$ Now we show that $d((B\setminus \overline{B_{2^{-3}}})\setminus N^{\frac{r}{2}}, \Sigma)>0$. Otherwise, since $\Sigma \subset N^{\frac{r}{2}}$ by definition, there exists a point $z\in \Sigma$ so that $(\frac{|z|r}{2})^{4-2n} \mu(B_{\frac{|z|r}{2}}(z))\leq \frac{\epsilon_0}{2}$ but $z\in \Sigma$ which is impossible by Lemma \ref{lem2.8}. Since $N^{\frac{r}{2}}\subset N^{2r}$, similar reasons show that 
$$d((B\setminus \overline{B_{2^{-3}}})\setminus N^{2r},N^{\frac{r}{2}})>0.$$ This finishes the proof.
\end{proof}

\subsubsection{Proof of Theorem \ref{Cut-Off} (III).}
Given a point $p\in \mathbb{C}^n\setminus\{0\}$, we can choose an $n-2$ dimensional complex linear subspace  $\C^{n-2}_{p}\subset\C^n$ that contains $\C\cdot p$. Then using the flat metric on $\C^n$, we can identify $\C^n$ with an orthogonal product $\C^{2} \times \C^{n-2}_p$ at $p$. 

\begin{defi}
We say a closed subset $S\subset \overline{B}$ admits a \emph{good cover} if $S \cap (\overline B _{2^{-1}}\setminus B_{2^{-2}})$ can be covered by finitely many open sets $U_{k}\subset B_{\frac{3}{4}}\setminus \overline{B_{\frac{3}{16}}}$ such that
\begin{itemize}
\item $ U_k =B^{2}_{\delta^k_2}\times B_{\delta^k_3}^{n-2}\subset \mathbb{C}^2 \times \mathbb{C}^{n-2}_{p_k}$ for some point $p_k\in \C^n\setminus \{0\}$, some choice of $\C^{n-2}_{p_k}$, and some $\delta_2^k,\delta_3^k>0$, where $B^2_{\delta^k_2}$ denotes the ball $\{|z|<\delta\}$ in $\C^2$ and $B^2_{\delta^k_3}$ denote the ball of radius $\delta_3^k$ centered at $p_k$ in $\C^{n-2}_{p_k}$;
\item $\emptyset\neq \overline{U_k}\cap S \subset V_k =B^2_{\delta^k_1}\times \overline{B}^{n-2}_{\delta^k_3}$ for some $\delta^k_1\in(0, \delta^k_2)$;
\end{itemize}
\end{defi}

\begin{lem}\label{lem2.10}
For any tangent cone $(A_\infty, \Sigma, \mu)$, $\Sigma$ admits a good cover. 
\end{lem}
\begin{proof}
By Lemma \ref{lem3.3}, we know $\Sigma$ is a codimension $2$ complex subvariety of $B\setminus 0$ which is $\C^*$ invariant. Then given any $p\in \Sigma \cap (\overline{B_{2^{-1}}}\setminus B_{2^{-2}})$, for a generic orthogonal projection $\rho_p$ to some $\C^{n-2}_{p}$ at $p$, $\rho_p^{-1}(y)\cap \Sigma$ consists of finitely many points for any $y\in B_{\delta^p_3}^{n-2}$ for some $\delta^p_3>0$. Then near $p$, one can easily construct a neighborhood $U_p$ of $p$ so that $U_p=B^{2}_{\delta^p_2}\times B^{n-2}_{\delta^p_3}\subset B_{\frac{3}{4}}\setminus \overline{B_{\frac{3}{16}}}$ for some $\delta^p_2, \delta^p_3>0$ and $\overline{U_p}\cap \Sigma \subset V_p$ where $V_p=B^2_{\delta_1^p}\times \overline{B^{n-2}_{\delta_3^p}}$ for some $\delta_1^p \in (0,\delta_2^p)$. Now we get an open cover of $\Sigma\cap (\overline{B_{2^{-1}}}\setminus B_{2^{-2}})$ given by $\cup_p U_p$. Since $\Sigma\cap (\overline{B_{2^{-1}}}\setminus B_{2^{-2}})$ is compact, one can get a finite subcover $\cup_k U_{p_k}$.
\end{proof}

\begin{prop}\label{prop2.13}
There exists $r_0\in (0, r_0']$ such that for any $r\in (0,  r_0]$, $ N^{2r}$ admits a good cover for all tangent cones $(A_\infty, \Sigma, \mu)$. 
\end{prop}
\begin{proof}
Otherwise, there exists a subsequence $r_i\searrow 0$ such that for each $r_i$ there exists $N^{2r_i}$ for some tangent cone $(A_{\infty}(i),\Sigma(i), \mu(i) )$ which does not admit a good cover. By using part (II) of Theorem \ref{Cut-Off}, for each $i$, there exists $A_{j_i}$ so that $N^{2r_{i}}\subset \Sigma^{4r_{i}}_{j_i}$. By passing to a subsequence, we can assume $\{A_{j_i}\}_i$ converge to some tangent cone $(A_\infty, \Sigma, \mu)$ and $\Sigma^{4r_i}_{j_i}$ converges to a closed subset of $\Sigma$. In particular, $\{N^{2r_i}\}_i$ converges to a closed subset of $\Sigma$. By Lemma \ref{lem2.10}, $\Sigma$ admits a good cover and we let $\cup_k U_k$ be the corresponding finite cover. Now we conclude that for $i$ large, $\cup_k U_k$ is also a good cover of $N^{2r_i}$, which is a contradiction. It suffices to verify the following
\begin{itemize}
\item $N^{2r_i}\cap( \overline{B_{2^{-1}}}\setminus B_{2^{-2}} )\subset \cup_k U_k$ for $i$ large. This is obvious since $\cup_k U_k$ is an open cover of $\Sigma \cap \overline{B_{2^{-1}}}\setminus B_{2^{-2}} $ and $\{N^{2r_i}\cap( \overline{B_{2^{-1}}}\setminus B_{2^{-2}})\}_i$ converge to a closed subset of $\Sigma \cap \overline{B_{2^{-1}}}\setminus B_{2^{-2}}$.   
\item $\overline{U_k}\cap N^{2r_i} \subset V_k$ for $i$ large. Otherwise, by passing to a subsequence and using the finiteness of $\{U_k\}_k$, we can assume for some fixed $k$, there always exists $z_{i} \in (\overline{U_k}\cap N^{2r_i})\setminus V_k$ for each $i$ and $z_i$ converges to $z\in \overline{U_k}\cap\Sigma$. Then $z\in V_k$ and thus $z_i\in V_k$ for $i$ large. Contradiction.
\end{itemize}
\end{proof}
As a direct corollary, we are ready to prove (III) for all $r\in(0, r_0]$. 
\begin{proof}[Proof of (III)]
By Proposition \ref{prop2.13}, $N^{2r}$ admits a good cover and let $\{U_k\}_k$ be the corresponding cover. So for each $k$
$$
\overline{U_k} \cap \Sigma \subset \overline{U_k}\cap N^{\frac{r}{2}}\subset \overline{U_k} \cap N^{2r} \subset V_k
$$  
where $\overline U_k =\overline{B^{2}_{\delta^k_2}}\times \overline{B_{\delta^k_3}^{n-2}}\subset \mathbb{C}^2 \times \mathbb{C}^{n-2}_{p_k}$ and $V_k=B^2_{\delta^k_1}\times \overline{B^{n-2}_{\delta^k_3}}$. Consider the projection $\rho_k: \C^2\times \mathbb{C}^{n-2}_{p_k} \rightarrow \C^{n-2}_{p_k}$.  By assumption, we have
$$(B^2_{\delta^k_2}\times \overline{B^{n-2}_{\delta^k_3}}) \cap \Sigma = \overline{U_k} \cap \Sigma \subset B^2_{\delta^k_1}\times \overline{B^{n-2}_{\delta^k_3}}$$
which implies $\rho_k^{-1}(y)\cap \Sigma \cap B^{2}_{\delta_2}$ is a compact complex analytic subvariety of $B^2_{\delta_2}$ and thus consists of finitely points for any $y\in \overline{B^{n-2}_{\delta_3}}$. For any $z\in \overline{N^{2r}\setminus N^\frac{r}{2}}\cap (\overline{B}_{2^{-1}}\setminus B_{2^{-2}})$, suppose $z\in U_k$, then $\rho^{-1}_k(\rho_k(z))\cap\Sigma \cap B^2_{\delta_2}$ consists of finitely many points which lie in $B^{2}_{\delta_1}$. As a result, one can easily find a flat holomorphic disk $D_z \subset U_k\cap B\setminus B_{2^{-2}}$ containing $z$ such that $D_z \cap \Sigma=\emptyset$ and $\p D_z \subset U_k\setminus V_k \subset (B_{\frac{3}{4}}\setminus\overline{B}_{2^{-2}})\setminus N^{2r}$. By varying the disk $D_z$, we can find an open neighborhood $V_z$ of $z$ so that for each $z' \in V_z$ there exists a flat holomorphic disk $D_{z'}\subset B\setminus B_{2^{-2}}$ so that $D_{z'}\cap \Sigma =\emptyset$ and $\p D_{z'} \subset (B_{\frac{3}{4}}\setminus\overline{B}_{2^{-2}})\setminus N^{2r}$. Furthermore, $\inf_{z'\in V_z}\min\{d(D_{z'}, \Sigma), d(\p D_z, \overline{B_{2^{-2}}})\}>0$. For $z\in (\overline{B}_{2^{-1}}\setminus B_{2^{-2}})\setminus N^{2r}$, it is obvious that one can do the same thing as above. As a result, we get an open cover $\cup_{z\in \overline{(\overline{B}_{2^{-1}}\setminus B_{2^{-2}})\setminus N^{\frac{r}{2}}}} V_z$ of $\overline{(\overline{B}_{2^{-1}}\setminus B_{2^{-2}})\setminus N^{\frac{r}{2}}}$.  Since $\overline{(\overline{B}_{2^{-1}}\setminus B_{2^{-2}})\setminus N^{\frac{r}{2}}}$ is compact, we can find a finite subcover $\cup_{z_i} V_{z_i}$. Let $C(r)=\min_{i} \inf_{z\in V_{z_i}} \min\{ d(D_{z_i},\sigma), d(\partial D_{z_i}, \partial \overline{B_{2^{-2}}})\}$. This finishes the proof.  
\end{proof}

\subsection{Convexity}\label{convexity}
In this section, we will refine the convexity result obtained in \cite{CS1}. Given a tangent cone $(A_\infty, \Sigma, \mu)$, suppose $W$ is a symmetric open subset of $(B\setminus \overline{B_{2^{-3}}})\setminus \Sigma$.  
\begin{defi}
A non-zero holomorphic section $s$ of $\E_\infty$ over $W$ is \emph{homogeneous} of degree $d$ if 
$$\nabla_{\p_r} s=dr^{-1}s. \footnote{Since $s$ is holomorphic, this is equivalent to  $\nabla_{J\p_r}s=\sqrt{-1}dr^{-1}s$.} $$
\end{defi}

The following has been observed in \cite{CS1} (see Lemma $3.2$ and the paragraph after Lemma $3.2$  in \cite{CS1}). 
\begin{lem}\label{descend}
Suppose $s$ is a homogeneous sections of $\E_\infty$ with degree $d$, then $d\in  ((\rk\E_\infty)!)^{-1}\mathbb{Z}.$ Furthermore, $s=\pi^*\underline s$ for some section $\underline s$ of some simple Hermitian-Yang-Mills cone summand.
\end{lem}

The following is a slight extension of  Proposition $3.6$ in \cite{CS1}. For the convenience of readers we include the proof here. 

\begin{prop}\label{Cauchy-Schwarz-Inequality}
Suppose $s\in H^{0}(W, \E_\infty)$ with $\int_{W}|s|^2 <\infty$, then 
$$
(\int_{W\cap (B_{2^{-1}}\setminus \overline{B_{2^{-2}}})}|s|^2)^2 \leq 
\int_{W\cap (B_{2^{-2}}\setminus \overline{B_{2^{-3}}})} |s|^2\cdot \int_{W\cap (B\setminus \overline{B_{2^{-1}}})}|s|^2.
$$ 
Furthermore, if $s$ is non-zero and the equality holds, then $s$ must be homogeneous. 
\end{prop}
\begin{proof}
Similar to the proof of Lemma $3.4$ in \cite{CS1}, we can write $s=\sum_d  s_d$ where $s_d$ is a homogeneous section of $\E_\infty$  over $W$ of degree $d\in ((\rk\E_\infty)!)^{-1}\mathbb{Z}$. Then we have
$$\int_{W\cap (B\setminus \overline{B_{2^{-1}}})}|s|^2=\sum_{d} \int_{W\cap (B\setminus \overline{B_{2^{-1}}})} |s_d|^2,$$
and 
$$\int_{W\cap (B_{2^{-1}}\setminus \overline{B_{2^{-2}}})}|s|^2=\sum_{d} 2^{-2d-2n}\int_{W\cap (B\setminus \overline{B_{2^{-1}}})} |s_d|^2,$$
and 
$$\int_{W\cap (B_{2^{-2}}\setminus \overline{B_{2^{-3}}})}|s|^2=\sum_{d} 2^{-4d-4n}\int_{W\cap (B\setminus \overline{B_{2^{-1}}})} |s_d|^2.$$
Now the conclusion follows from the general Cauchy-Schwarz inequality. 
\end{proof}

Now given a saturated subsheaf $\F\subset \E$ (which can be trivial),   we denote by $\pi_\F: \E\rightarrow \F$ the pointwise orthogonal projection with respect to the admissible Hermitian-Einstein metric $H$ and let $\pi_\F^{\perp}=\Id-\pi_\F$. Note $\pi_\F$ is only defined away from $\text{Sing}(\E/\F)$. In the following we shall work under the following hypothesis, and in our later application this hypothesis will always be verified. 

\

\emph{ $**$: Given any subsequence $\{j_i\}$, by passing to a further subsequence, $\{A_{j_i}\}_i$ converges to a tangent cone $(A_\infty, \Sigma, \mu)$, and the corresponding pull-backs of $\pi_{\F}$ under the map $z\mapsto 2^{-j_i}z$ converges locally smoothly to a projection map $\pi_\infty$ on $\E_\infty$ away from $\Sigma$. Furthermore, $\pi_\infty$ is exactly the orthogonal projection onto a HYM cone direct summand $\F_\infty\subset \E_\infty$ (away from $\Sing(\E_\infty/\F_\infty)$).}

\

Given any fixed $r\in (0, r_0]$, for any \emph{smooth} section $\sigma$ of $\E$ over $(B_{2^{-j-1}}\setminus \overline{B_{2^{-j-2}}})\setminus  2^{-j}E^r_{j}$ (See Equation \ref{Ejr} for definition of $E^r_j$), let 
\begin{equation}\label{eqn2.1}
\|\sigma\|^r_{j}=2^{jn}(\int_{(B_{2^{-j-1}}\setminus \overline{B_{2^{-j-2}}})\setminus 2^{-j}E^r_{j}} |\sigma|^2)^{\frac{1}{2}}.
\end{equation}

\begin{prop}\label{prop2.18}
Given any $r\in (0, r_0]$ and $\lambda\notin ((rank \E)!)^{-1}\mathbb{Z}$, there exists $j_0=j_0(r, \lambda)$ such that for all $j\geq j_0$, if  $s\in H^{0}(B_{j-1},\E)$ satisfies
$$\|\pi_\F^{\perp} s\|^r_{j} > 2^{-\lambda}\|\pi_\F^{\perp}s\|^r_{j-1},$$ 
then 
$$\|\pi_\F^{\perp} s\|^r_{j+1} \geq 2^{-\lambda}\|\pi_\F^{\perp}s\|^r_{j}.$$ 
\end{prop}
\begin{proof}
Otherwise, there exists a sequence of holomorphic sections $s_{j_i}\in H^{0}(B_{j_i-1}, \E)$ so that 
$$
\|\pi_\F^{\perp}s_{j_i}\|^r_{j_i}=1,
$$ 
and 
$$
\|\pi_\F^{\perp}s_{j_i}\|^r_{j_i-1}<2^{\lambda},
$$
 but 
$$
\|\pi_\F^{\perp}s_{j_i}\|^r_{j_i+1} < 2^{-\lambda}.
$$
By passing to a subsequence, we can assume $\{A_{j_i}\}_{i}$ converges to a tangent cone $(A_\infty, \Sigma, \mu)$, and the statements in $(**)$ hold. By passing to a further subsequence, we may assume $\{E^r_{j_i-1}\}_i, \{E^r_{j_i}\}_i$ and $\{E^r_{j_i+1}\}_i$ converge to $W^r_0, W^r_1$ and $W^r_2$ respectively, which are all symmetric. We then denote 
$$W^r:=2W^r_0\cup W^r_1\cup 2^{-1}W^r_2. $$
Let $\sigma_{j_i}=(2^{-j_i})^{*}\pi_\F^{\perp}s_{j_i}$. Then we have
$$\int_{(B\setminus \overline{B_{2^{-3}}})\setminus \Sigma^r_{j_i}}|\sigma_{j_i}|^2 < 1+2^{\lambda}+2^{-\lambda},$$
and
$$
((2^{-j_i})^{*}\pi^{\perp}_{\F})\circ \bp_{A_{j_i}} \sigma_{j_i}=0
$$
over $(B\setminus \overline{B}_{2^{-3}})\setminus \Sigma^{r}_{j_i}$. Hence $\sigma_{j_i}$ is a holomorphic section of $\F^{\perp}_{j_i}=(2^{-j_i})^*\F^\perp$ over $(B\setminus \overline{B}_{2^{-3}})\setminus \Sigma^{r}_{j_i}$ with uniformly bounded $L^2$ norm. Since we have smooth convergence of $(2^{-j_i})^*\pi_{\F}$ locally away from $\Sigma$, by standard elliptic theory, after passing to a subsequence, we can assume $\{\sigma_{j_i}\}_i$ converges to $\sigma_\infty$ locally smoothly over $(B\setminus \overline{B_{2^{-3}}})\setminus W^r$. Then $\sigma_\infty$  is a holomorphic section of $\F^{\perp}_\infty$ over $(B\setminus \overline{B_{2^{-3}}})\setminus W^r$ satisfying
$$
\int_{(B\setminus \overline{B_{2^{-1}}})\setminus 2W^r_0} |\sigma_\infty|^2 \leq 2^{\lambda},
$$
and 
$$
\int_{(B_{2^{-1}}\setminus \overline{B_{2^{-2}}})\setminus W^r_1} |\sigma_\infty|^2 \leq 1,
$$
and 
$$
\int_{(B_{2^{-2}}\setminus \overline{B_{2^{-3}}})\setminus 2^{-1}W^r_2} |\sigma_\infty|^2 \leq 2^{-\lambda}.
$$
Let $V^r_1:=W^r_0\cup W^r_1 \cup W^r_2$ and $V^r=2V^r_1 \cup V^r_1 \cup 2^{-1} V^r_1$. By Theorem \ref{Cut-Off} (I), we have
 $$m(V^r_1\setminus W^r_l)=m(W^r_l\setminus V^r_1)=0$$
for $l=0,1,2$. Then we have 
$$
\int_{(B\setminus \overline{B_{2^{-1}}})\setminus 2V^r_1} |\sigma_\infty|^2 \leq 2^{\lambda},
$$
and 
$$
\int_{(B_{2^{-1}}\setminus \overline{B_{2^{-2}}})\setminus V^r_1} |\sigma_\infty|^2 \leq 1,
$$
and 
$$
\int_{(B_{2^{-2}}\setminus \overline{B_{2^{-3}}})\setminus 2^{-1}V^r_1} |\sigma_\infty|^2 \leq 2^{-\lambda}.
$$
\begin{clm}\label{clm2.19}
$\int_{(B_{2^{-1}}\setminus \overline{B_{2^{-2}}}) \setminus V^r_1} |\sigma_\infty|^2=1.$
\end{clm}
Given this claim, by applying Proposition \ref{Cauchy-Schwarz-Inequality} to $\sigma_\infty$ over $(B\setminus \overline B_{2^{-3}})\setminus V^r$, we know $\sigma_\infty$ is a nonzero homogeneous section of $\E_\infty$ of degree $\lambda$ over  $(B\setminus \overline B_{2^{-3}})\setminus V^r$. This contradicts with our hypothesis that $\lambda \notin ((\rk \E)!)^{-1} \mathbb{Z}$. 
\end{proof}

\begin{proof}[Proof of Claim \ref{clm2.19}]
By assumption we know $N^{\frac{r}{2}}\subset \Sigma^{r}_{j_i}\cap \Sigma^r_{j_i-1}\cap \Sigma^r_{j_i+1}$ for $i$ large, so $N^{\frac{r}{2}}\subset V^r_1$. Then we have
$$
\|\sigma_{j_i}\|_{L^{\infty}((\overline{B_{2^{-1}}}\setminus B_{2^{-2}}) \setminus E_{j_i}^{r})}\leq \|\sigma_{j_i}\|_{L^{\infty}((\overline{B_{2^{-1}}}\setminus B_{2^{-2}}) \setminus N^{\frac{r}{2}})} .
$$
It suffices to prove that there exists $C=C(r)$ independent of $i$ such that for all $i$ large
\begin{equation}\label{eqn2.4}
\|\sigma_{j_i}\|_{L^{\infty}((\overline{B_{2^{-1}}}\setminus B_{2^{-2}}) \setminus N^{\frac{r}{2}})} \leq C.
\end{equation}
Indeed, given this, since $|\sigma_{j_i}|^2$ converge to $|\sigma|^2$ pointwisely over $(B_{2^{-1}}\setminus \overline{B_{2^{-2}}})\setminus E^r_{j_i}$, the conclusion follows from  Lebesgue's dominated convergence theorem as follows 
$$
\begin{aligned}
\int_{(B_{2^{-1}}\setminus \overline{B_{2^{-2}}}) \setminus V^r_1} |\sigma_\infty|^2
&=\int_{(B_{2^{-1}}\setminus \overline{B_{2^{-2}}}) \setminus W^r_1}|\sigma_\infty|^2 
\\
&=\int_{(B_{2^{-1}}\setminus \overline{B_{2^{-2}}}) } \chi|\sigma_\infty|^2
\\
&=\lim_i \int_{(B_{2^{-1}}\setminus \overline{B_{2^{-2}}}) } \chi_{j_i}|\sigma_{j_i}|^2\\
&=\lim_i \int_{(B_{2^{-1}}\setminus \overline{B_{2^{-2}}}) \setminus E_{j_i}^r} |\sigma_{j_i}|^2
&=1.  
\end{aligned}
$$
where all the equalities follow from the definition except the third one follows from Lebesgue's dominated convergence theorem. Here $\chi$ denotes the characteristic function over $B_{2^{-1}}\setminus \overline{B_{2^{-2}}}$ given by $(B_{2^{-1}}\setminus \overline{B_{2^{-2}}}) \setminus W^r_1$ and  $\chi_{j_i}$ denotes the charateristic function over $B_{2^{-1}}\setminus \overline{B_{2^{-2}}}$ given by $(B_{2^{-1}}\setminus \overline{B_{2^{-2}}}) \setminus E^r_{j_i}$, which converge to $\chi$ pointwisely by assumption.  

Now we prove inequality  \ref{eqn2.4}. By Theorem \ref{Cut-Off} (III), there exists a constant $C=C(r)>0$ so that for any $z\in \overline{(\overline{B_{2^{-1}}}\setminus B_{2^{-2}}) \setminus N^{\frac{r}{2}}}$, there exists a flat holomorphic disk so that $D_z \subset B_{\frac{3}{4}}\setminus B_{2^{-2}}$ with $D_z\cap \Sigma =\emptyset$ and $\p D_z\subset (B_{\frac{3}{4}}\setminus \overline{B_{2^{-2}}})\setminus N^{2r}$ and $d(D_z, \Sigma)\geq C>0$. Since $A_{j_i}$ converges to $A_\infty$ locally smoothly over $(B\setminus\{0\})\setminus \Sigma$, there exists a constant $C_1=C_1(r)$ so that for any $z\in B_{\frac{3}{4}}\setminus \overline{B_{\frac{3}{16}}}$ with $d(z,\Sigma) \geq C$
\begin{equation}\label{eqn2.5}
|F_{A_{j_i}}|(z)\leq C_1<\infty, \ 
|\bp_{A_{j_i}} (2^{-j_i})^*\pi_{\F}| \leq C_1<\infty.
\end{equation}
Then \eqref{eqn2.4} follows from the following 
\begin{lem}\label{lem2.20}
$|\sigma_{j_i}(z)| \leq C_2\cdot (\|\sigma_{j_i}\|_{L^{\infty}(\p D_z)}+1)$ where $C_2=C_2(r)$ is a constant independent of $i$. 
\end{lem}

\noindent Indeed, given this, by Theorem \ref{Cut-Off} (III), there exists $C_3=C_3(r)$ so that $d(\partial D_z, \partial \overline{B_{2^{-2}}})>0$ and we have 
\begin{equation}\label{eqn}
\|\sigma_{j_i}\|_{L^{\infty}((\overline{B_{2^{-1}}}\setminus B_{2^{-2}}) \setminus N^{\frac{r}{2}})}\leq C\cdot (\|\sigma_{j_i}\|_{L^{\infty}((B_{\frac{3}{4}}\setminus \overline{B_{2^{-2+\frac{C_3}{2}}}})\setminus N^{2r})}+1).
\end{equation}
By  Theorem \ref{Cut-Off} (II), we have 
$$d((B\setminus \overline{B}_{2^{-3}})\setminus N^{\frac{r}{2}}, \Sigma)>0$$ 
which implies $\{A_{j_i}\}_i$ converge to $A_\infty$ uniformly over a  neighborhood of $\overline{(B\setminus \overline{B}_{2^{-3}})\setminus N^{\frac{r}{2}}}$.
By Theorem \ref{Cut-Off} (II) again, we have 
$$\liminf_{i}d((B\setminus \overline{B}_{2^{-3}})\setminus N^{2r}, \Sigma_{j_i}^r)>0,$$
which implies that $(B_{\frac{3}{4}}\setminus \overline{B_{2^{-2+\frac{C_3}{2}}}})\setminus N^{2r}$ lies in the interior of $\overline{(B\setminus \overline{B_{2^{-2}}})}\setminus \Sigma^r_{j_i}$ for $j_i$ large and also has definite distance to the boundary of $\overline{(B\setminus \overline{B_{2^{-2}}})}\setminus \Sigma^r_{j_i}$ for $j_i$ large. As a result, using the uniform $L^2$ bound of $\{\sigma_{j_i}\}_{j_i}$ over $\overline{(B\setminus \overline{B_{2^{-2}}})}\setminus \Sigma^r_{j_i}$, by elliptic interior estimate, we get a uniform $L^\infty$ bound $\{\sigma_{j_i}\}_{j_i}$ over $(B\setminus \overline{B}_{2^{-3}})\setminus N^{2r}.$ This finishes the proof.  
\end{proof}

\begin{proof}[Proof of Lemma \ref{lem2.20}]
Let $\nabla:=  A_{\F^\perp_{j_i}}|_{D_z}$ which has curvature form  
$$
F_\nabla=(F_{A_{j_i}}-(\bp_{A_{j_i}} (2^{-j_i})^*\pi_\F )^*\wedge \bp_{A_{j_i}} (2^{-j_i})^*\pi_\F)|_{D_z}.
$$
Since $\sigma_{j_i}|_{D_z}$ is a holomorphic section of $\F^{\perp}_{j_i}|_{D_z}$,  we have 
\begin{equation}\label{eqn2.6}
\Delta_{D_z} \log (|\sigma_{j_i}|_{D_z}|^2+1) \geq -|F_\nabla|\geq -2C_1
\end{equation}
where the second inequality follows from \ref{eqn}. For the first inequality, we first identify $D_z$ with $\{t\in \C: |t|<\delta_z\}$ where $\delta_z$ is the radius of $D_z$ and by  a direct calculation, we have 
$$
\Delta \log (|\sigma|^2+1)=\frac{<\sigma, \nabla_{\bar{\p}_t}\nabla_{\p_t}\sigma>}{|\sigma|^2+1}+\frac{|\nabla_{\p_t}\sigma|^2}{|\sigma|^2+1}-\frac{< \nabla_{\p_t}\sigma, \sigma>< \sigma, \nabla_{\p_t}\sigma>}{(|\sigma|^2+1)^2}
$$
The difference of the last two terms is non-negative by Cauchy-Schwarz inequality, and for the first term we have
$$
\frac{<\sigma, \nabla_{\bar{\p}_t}\nabla_{\p_t}\sigma>}{|\sigma|^2+1}=\frac{<\sigma, F_{\nabla}(\frac{\p}{\p \bar t}, \frac{\p}{\p t})\sigma>}{|\sigma|^2+1}
\geq -|F_\nabla|.
$$
 As a result, we get
$$\Delta_{D_z}(\log (|\sigma_{j_i}|_{D_z}|^2+1)+2C_1|t|^2)\geq 0$$
Now the conclusion follows from the maximum principle. 
\end{proof}

\begin{rmk}\label{rmk2.27}
Note from the proof we have that for $i$ large $\|\sigma_{j_i}\|_{L^{\infty}((\overline{B_{2^{-1}}}\setminus B_{2^{-2}}) \setminus N^{\frac{r}{2}})}$ can be controlled by only $\|\sigma_{j_i}\|_{L^2((B\setminus \overline{B_{2^{-2}}}) \setminus \Sigma^r_{j_i})}$. 
\end{rmk}

\begin{prop}\label{Cor2.11}
Given a local holomorphic section $s$ of $\E$ in a neighborhood of $0$, the following is a well-defined number in $((\rk \E)!)^{-1}\mathbb{Z}\cup \{\infty\}$
$$
d^r_{\F}(s):=\lim_{j\rightarrow\infty}\frac{\log\|\pi^{\perp}_{\F}s\|^r_j}{-j\log 2}
$$
for any $r\in (0,r_0]$. 
\end{prop}
\begin{proof}
Denote 
$$b_j=\frac{\log \|\pi^{\perp}_\F s\|^r_j- \log \|\pi^{\perp}_\F s\|^r_{j-1}}{-\log2}.$$  By Proposition \ref{prop2.18}, for any $ \lambda\notin ((rank\E)!)^{-1} \mathbb{Z}$, there exists $j_0=j_0(\lambda)$ so that for $j\geq j_0$, if  $b_j \leq \lambda$, then $b_{j+1} \leq \lambda$. Then it follows directly that 
$$\liminf_{j\rightarrow\infty} b_j = \limsup_{j\rightarrow\infty} b_{j+1}.$$
 Indeed, if not, then there exists $\lambda \notin ((rank\E)!)^{-1} \mathbb{Z}$ so that 
 $$\liminf_{j\rightarrow\infty} b_j<\lambda<\limsup_{j\rightarrow\infty} b_{j+1}.$$
 In particular, we know that there exists a subsequence $b_{j_i}$ so that $b_{j_i} \leq \lambda$. Fix $j_{i_0}\geq j_0$ and by assumption, we know for any $j\geq j_{i_0}$, $b_j \leq \lambda$. By proposition \ref{prop2.18}, we have $\limsup_{j\rightarrow\infty} b_j \leq \lambda$ which is a contradiction. It follows from this that $\lim_{j\rightarrow\infty} b_j$ is well defined in $\mathbb{R} \cup \{\pm \infty\}$. Now we rule out $-\infty$ and show it is actually well-defined in $((rank \E)!)^{-1}\mathbb{Z} \cup \{\infty\}$. Fix a subsequence $\{j_i\}$ so that $\lim_iA_{j_i}=(A_\infty, \Sigma, \mu)$. 
 
\begin{clm}\label{clm2.29}
If $\lim_{j\rightarrow\infty} b_j<\infty$, $\{\sigma_{j_i}:=\frac{(2^{-j_i})^{*}\pi^{\perp}s}{\|\pi^{\perp}s\|^r_{j_i}}\}_i$ converges to a nontrivial homogeneous holomorphic section $\sigma_\infty$ over $(B_{2^{-1}}\setminus \overline B_{2^{-2}}) \setminus V^r$.
\end{clm}   
Given this claim, we must have $\lim_{j\rightarrow\infty} b_j>-\infty$. If it is not $\infty$, by Proposition \ref{prop2.18}, we know $\{\sigma_{j_i}:=\frac{(2^{-j_i})^{*}\pi^{\perp}s}{\|\pi^{\perp}s\|^r_{j_i}}\}_i$ converge to a homogeneous sections of degree equal to $\lim_{j\rightarrow\infty} b_j$ which implies 
$$\lim_{j\rightarrow\infty} \frac{\sum^j_{l=1} b_l}{j}=\lim_{j\rightarrow\infty} b_j \in ((rank \E)!)^{-1}\mathbb{Z}$$
by Lemma \ref{descend}. This finishes the proof.
\end{proof}

\begin{proof}[Proof of Claim \ref{clm2.29}]
Otherwise for any $\lambda\notin  ((rank \E)!)^{-1}\mathbb{Z}$, by Proposition \ref{prop2.18}, for $j_i$ large we have $\|\sigma_{j_i}\|_{L^2((B\setminus \overline{B_{2^{-1}}})\setminus E^{r}_{j_i-1})} \leq 2^{\lambda}$. In particular, we have a uniform bound on $\|\sigma_{j_i}\|_{L^2((B\setminus \overline{B_{2^{-2}}})\setminus \Sigma^{r}_{j_i-1})}$ given by $1+2^{\lambda}$. Let $\sigma_\infty$ denote the limit over $(B\setminus \overline{B_{2^{-2}}})\setminus (E^r_{j_i}\cup E^r_{j_i-1})$. By Remark \ref{rmk2.27}, we have a uniform bound on $\|\sigma_{j_i}\|_{L^{\infty}((\overline{B_{2^{-1}}}\setminus B_{2^{-2}}) \setminus N^{\frac{r}{2}})}$ which implies the strong convergence of $\{\sigma_{j_i}\}_i$ over $(B_{2^{-1}}\setminus \overline{B_{2^{-2}}})\setminus E^r_{j_i}$and thus the limit $\sigma_\infty$ must be nontrivial by the argument in the proof  Proposition \ref{prop2.18} . However, we also know that $\|\sigma_\infty\|_{L^2((B\setminus \overline{B_{2^{-1}}})\setminus W^r_2)}\leq 2^{\lambda}$ for any $\lambda\notin  ((rank \E)!)^{-1}\mathbb{Z}$ where $W^r_2$ denote the limit of $E^r_{j_i-1}$. In particular, $\sigma_\infty =0$ over $(B\setminus \overline{B_{2^{-1}}})\setminus W^r_2$, which is impossible. 
\end{proof}

\begin{prop}\label{prop2.23}
Given a local section $s$  of $\E$ near the origin, if $d^r_{\F}(s)$ is finite for some $r\in (0,r_0]$, then $\{\frac{(2^{-j_i})^*\pi^{\perp}_{\F}s}{\|\pi^{\perp}_{\F}s\|^r_{j_i}}\}_{i}$ converges to a non-trivial  homogeneous section $\sigma_\infty$ of degree $d^r_{\F}(s)$ of $\F^\perp_\infty$ over $(B_{2^{-1}}\setminus \overline B_{2^{-2}})\setminus \Sigma$,  which extends to a holomorphic section of $\F^\perp_\infty$ defined over $B_{2^{-1}}\setminus \overline B_{2^{-2}}$.
\end{prop}
\begin{proof}
If $d^r_{\F}(s)<\infty$, by passing to a subsequence, it follows from the proof of Proposition \ref{prop2.18} that $\{\sigma_{j_i}:=\frac{(2^{-j_i})^{*}\pi^{\perp}s}{\|\pi^{\perp}s\|^r_{j_i}}\}_i$ converges to a nontrivial homogeneous holomorphic section $\sigma_\infty$ of degree $d^r_{\F}(s)$ over $(B\setminus \overline B_{2^{-3}}) \setminus V^r$. Furthermore, we also have 
\begin{equation}\label{eqn2.7}
\|\sigma_{j_i}\|_{L^{\infty}((B_{2^{-1}}\setminus \overline{B}_{2^{-2}})\setminus N^{\frac{r}{2}})} \leq C(\|\sigma_{j_i}\|_{L^\infty(B_{\frac{3}{4}}\setminus \overline{B}_{2^{-2}})\setminus N^{2r})}+1).
\end{equation}
for some $C=C(r)$. By definition, we know $d^r_{\F}(s)$ is decreasing when $r\rightarrow 0$. Hence we have $d^{\frac{r}{2}}_{\F}(s)<\infty$ which implies $\{ \frac{(2^{-j_i})^{*}\pi^{\perp}s}{\|s\|^{\frac{r}{2}}_{j_i}}\}$ converges to a homogeneous section $\sigma'_\infty$. Then by Equation (\ref{eqn2.7}), we can assume $\{\sigma_{j_i}\}_{i}$ converges to a nontrivial homogeneous holomorphic section of $\F^\perp_\infty$ over $(B\setminus \overline{B}_{2^{-3}})\setminus V^{\frac{r}{2}}$ which is a multiple of $\sigma'_\infty$. By repeating this process for $2^{-l}r$ inductively on $l\in \mathbb{Z}_{\geq 0}$ and passing to a diagonal subsequence, we can assume $\{\sigma_{j_i}\}_i$ converges to a nontrivial homogeneous holmorphic section $\sigma_\infty$ of $\F^{\perp}_\infty$ over $(B_{2^{-1}}\setminus \overline{B}_{2^{-2}})\setminus \Sigma$.  Now it remains to show that $\sigma_\infty$ extends to be a holomorphic section of $\F^{\perp}_\infty$ over $B_{2^{-1}}\setminus \bar{B}_{2^{-2}}$. Since $B_{2^{-1}}$ is a precompact Stein open set in $B$, we can find a finite resolution of $(\F^\perp_\infty)^{*}$ over $B_{2^{-1}}$ as 
$$
\O^{n_1} \rightarrow \O^{n_2} \rightarrow (\F^\perp_{\infty})^*\rightarrow 0.
$$
By taking the dual of the above exact sequence, we have the following exact sequence over $B_{2^{-1}}$
$$
0\rightarrow \F^\perp_\infty \rightarrow \O^{n_2} \rightarrow \O^{n_1}.
$$
Then we can view $\sigma_\infty$ as a holomorphic section of $\O^{n_2}$ over $(B_{2^{-1}}\setminus \overline B_{2^{-2}})\setminus \Sigma$. Since $\Sigma$ has Hausdorff codimension $4$, $\sigma_\infty$ extends to be a  holomorphic section of $\O^{n_2}$ over $B_{2^{-1}}\setminus \overline B_{2^{-2}}$ (see Lemma $3$ in \cite{Shiffman}).  This finishes the proof.
\end{proof}

\section{Uniqueness of tangent cone connections}
In this section, we will prove Theorem \ref{main}. The arguments are similar to that in \cite{CS1}, given Theorem \ref{Cut-Off}. The main difference is to replace the $L^2$ orthogonal projection in \cite{CS1} by  pointwise orthogonal projection, and for the convenience of readers, we reproduce the induction arguments to make this paper self-contained. To make the argument clear, we first deal with the case that $\underline \E$ is semistable in Section \ref{semistable}, and in Section \ref{unstable} we prove the case when $\underline \E$ is unstable. The technical part of the second case involves the construction of a good comparison metric, which has already been overcome in \cite{CS1}.

\subsection{Semistable Case}\label{semistable}
Assume $\underline \E$ is semistable and fix a Seshadri filtration for $\underline \E$ as 
$$
0=\underline \E_0 \subset \underline \E_1 \subset \underline\E_2 \subset \cdots \underline\E_m=\underline\E.
$$   
So 
$$Gr^{HNS}(\underline\E)\simeq \bigoplus_{p=1}^m \underline \E_p/\underline \E_{p-1}$$
Denote $\E_i=\pi^* \underline \E_i$. Theorem \ref{main} follows from 
\begin{thm}\label{thm3.3}
$(\E_\infty, A_\infty)$ is isomorphic to the natural Hermitian-Yang-Mills cone connection on $\psi_*\pi^*(Gr^{HNS}(\underline\E))^{**}$. Moreover, $\text{Sing}(\psi_*\pi^*(Gr^{HNS}(\underline\E))) \subset\Sigma$. 
\end{thm}

By tensoring with $\O(k)$ for $k$ large, we may assume the following for each $p\geq 1$
\begin{itemize}
\item  $\underline\E_p$ and $\underline \E_p/\underline \E_{p-1}$  are globally generated;
\item The following sequence is exact 
$$
0\rightarrow H^{0}(\C\P^{n-1}, \underline \E_{p-1}) \rightarrow H^{0}(\C\P^{n-1}, \underline \E_p)\rightarrow H^{0}(\C\P^{n-1},\underline \E_p/\underline{\E}_{p-1}) \rightarrow 0.
$$
\end{itemize}
Denote $HG_p :=\{\pi^* \underline s: \underline s\in H^0(\C\P^{n-1}, \underline \E_p)\}$. Then we have a filtration 
$$
0\subset HG_1 \subset HG_2\subset \cdots HG_m=HG.
$$

Denote $n_p:=\dim_{\C} HG_p/HG_{p-1}.$ For $p\geq 0$, let $\pi_p$ be the pointwise orthogonal projection from $\E$ to $\E_p$ with respect to the unknown metric $H$ and  $\pi_p^{\perp}=\Id-\pi_p$ which denotes the projection from $\E$ to the orthogonal complement of $\E_p$. Note $\pi_p$ and $\pi_p^{\perp}$ are both well-defined away from $\text{Sing}(\E/\E_p)$. Fix a basis  $\{\sigma_{p,l}| 1\leq p\leq m, 1\leq l\leq n_p\}$  of $HG$ so that $\{\sigma_{p,l}|1\leq l\leq n_p, 1\leq p\leq q\}$ form a basis for $HG_{q}$ for any $1\leq q\leq m$. For each $(p,l)$, we denote 
$$\sigma^j_{p,l}:=(2^{-j})^*(\pi_{p-1}^{\perp} \sigma_{p,l}). $$ 
We view these as  smooth sections of $\E$ defined on $B^*\setminus\text{Sing}(\E/\E_p)$.

Fix any $r\in (0,r_0]$, where $r_0$ is given in Theorem \ref{Cut-Off}.  Denote
 $$M^j_p=\sup_{1\leq l \leq n_p}\|\pi^{\perp}_{p-1}\sigma_{p,l}\|^r_{j}$$
  where $\|\cdot\|^r_{j}$ is defined as in Equation (\ref{eqn2.1}). The following is our starting point (compare Lemma 3.11 in \cite{CS1})
\begin{lem} \label{d0}
For any $s=\pi^{*}\underline s$, where $\underline s\in H^0(\C\P^{n-1}, \underline \E)$, we have $d^r_0(s)=\mu(\underline \E)$, where $d^r_0$ is the degree defined for the trivial subsheaf. 
\end{lem}

\begin{proof}
By Theorem $10.13$ in \cite{Kobayashi},  for any $\epsilon>0$, there exists a smooth Hermitian metric on $\underline H_\epsilon$ on $\underline{\E}$ so that 
$$\|\Lambda_{\omega_{FS}} F_{(\underline H_\epsilon, \bp_{\underline \E})}-\mu(\underline \E)Id\|_{L^{\infty}(\C\P^{n-1})}\leq \epsilon.$$
Let $H_\epsilon=|z|^{2\mu(\underline \E)}\pi^* \underline H_\epsilon$. By Lemma \ref{curvature}, since $|F_H|+|F_{H_\epsilon}|\in L^2(B^*)$, we know $ \log \Tr H_\epsilon^{-1}H \in L^{\frac{2n}{n-1}}(B^*)$ 
and $ \log \Tr H^{-1}H_\epsilon \in L^{\frac{2n}{n-1}}(B^*)$. By applying Lemma \ref{pde} with $g=\log \Tr H_\epsilon^{-1}H$ and $\log \Tr H^{-1}H_\epsilon$, we get
$$
C_\epsilon |z|^{2\epsilon} H_\epsilon\leq H \leq C^{-1}_\epsilon |z|^{-2\epsilon}H_\epsilon
$$
for some constant $C_\epsilon>0$. By definition, we have
$$d^r_0(s)=\lim_{j\rightarrow\infty} \frac{\log \|s\|^r_j}{-j\log2}\leq \lim_{j\rightarrow\infty} \frac{\log \int_{(B_{2^{-j-1}}\setminus \overline{B_{2^{-j-2}}})\setminus 2^{-j} E^r_j} |z|^{-2\epsilon}|s|^2_{H_\epsilon}}{-j\log2} \leq \mu(\underline \E)+\epsilon.$$
Similarly, $d_0^r(s) \geq \mu(\underline \E)-\epsilon$. By letting $\epsilon$ go to zero, we have $d_0^r(s)=\mu(\underline \E)$. 

\end{proof}

Now suppose $(A_\infty, \Sigma, \mu)$ is a tangent cone of $A$ given by the limit of a subsequence $\{A_{j_i}\}_i$. We shall prove the following statements by induction on $p\geq 1$.   Theorem \ref{thm3.3} is a direct corollary of these statements. 
\begin{itemize}
\item[$(a)_p$.] There is a simple HYM cone direct summand $\S_p$ of $\E_\infty$ which is isomorphic to $(\E_p/\E_{p-1})^{**}$ so that $\S_p \perp \S_{k}$ for any $k<p$ (We take $\S_0=0$ here);
\item[$(b)_p$.] $\Sing(\E_{p}/\E_{p-1})\cup \Sing(\E/\E_p)\subset \Sigma$, and   over $(B_{2^{-2}}\setminus B_{2^{-1}})\setminus \Sigma$, $\{\pi^{j_i}_p\}_{i}$ converges locally smoothly to the limit projection $\pi^\infty_p: \E_\infty \rightarrow \E_\infty$ given by $\oplus_{k\leq p} \S_k \subset \E_\infty$.  Here $\pi^{j_i}_p=(2^{-j_i})^*\pi_p$. In particular, the hypothesis $**$ in section \ref{convexity}  holds with $\E_p$ which enable us to run the convexity result.
We also denote by $(\E^{\infty}_p)^\perp$  the HYM cone direct summand of $\E_\infty$ so that 
$$\E_\infty=\oplus_{k\leq p} \S_k \oplus (\E^{\infty}_p)^\perp.$$ 
\item[$(c)_p$.] If $p<m$ then $d^r_{\E_p}(\sigma_{p+1,l})=\mu(\underline \E)$ for any $1\leq l\leq n_{p+1}$. Here $d^r_{\E_p}$ is well-defined due to $(b)_p$. 
\end{itemize}

\textbf{Base case $p=1$:}

\

\begin{itemize}
\item[$(a)_1$]  By Lemma \ref{d0}, after passing to further subsequence we may assume $\{\frac{\sigma^{j_i}_{1,l}}{M_1^{j_i}}\}_i$ converges to a holomorphic homogeneous section $\sigma_{1,l}^\infty$ of degree $\mu(\underline \E)$ away from $\Sigma$ for any $1\leq l\leq n_1$ and at least one of the limits is non-zero.  By assumption we have the following exact sequence of coherent sheaves 
$$
0\rightarrow R_1 \rightarrow \O^{\oplus n_1} \xrightarrow{\phi_1} \E_1\rightarrow 0
$$
where 
$$\phi_1(z)(a_1,\cdots, a_{n_1})=\sum_{ l=1}^{ n_1} a_{l} \sigma_{1,l}(z).$$
 Away from $\Sing(\E/\E_1)$, $\E_1$ can be viewed as a vector sub-bundle of $\E$. For $z\notin \Sigma\cup \text{Sing}(\E/\E_1)$, we define a vector bundle homomorphism
$$
\phi_1^\infty: \O^{\oplus n_1}\rightarrow \E_\infty
$$
by 
$$\phi_1^{\infty}(z)(a_1,\cdots, a_{n_1})=\lim_{i\rightarrow\infty}{(M^{j_i}_{1})}^{-1}{\sum_{l=1}^{n_1} a_{l} \sigma^{j_1}_{1,l}(z)}.$$
 If  $(a_1,\cdots, a_{n_1})$ is in the fiber of $(R_1)_z$, then by definition, we have $\sum_{l=1}^{n_1} a_{l} \sigma_{1,l}(z)=0$, hence $\sum_{l=1}^{n_1} a_{l} \sigma_{1,l}^{j_i}(z)=0$, which implies 
 $$\phi^{\infty}_1(z)(a_1,\cdots, a_{n_1})=0$$ and $\phi_1^\infty$ descends to a homomorphism away from $\Sigma$
$$
\psi_1: \E_1\simeq\O^{n_1}/R_1\rightarrow \E_\infty
$$
which satisfies $\psi_1(z)(\sigma_{1,l}(z))=\sigma^{\infty}_{1,l}(z)$. Let $Q_1$ be the maximal simple HYM cone  summand of $\E_\infty$ which contains the image of $\psi_1$. Note $Q_1$ is locally free away from $\Sigma$. Since $d(\sigma^{\infty}_{1,l})=\mu(\underline Q_1)$, by Lemma \ref{descend}, $\psi_1$ descends to be a nontrivial map defined over $\C\P^{n-1}\setminus (\pi(\Sigma)\cup \text{Sing}(\underline\E/\underline\E_1))$
$$
\underline \psi_1: \underline \E_1 \rightarrow \underline Q_1,
$$
where $\mu(\underline \E_1)=\mu(\underline \E)=\mu(\underline Q_1)$. By Lemma $3$ in \cite{Shiffman}, $\underline \psi_1$ extends to a sheaf homomorphism over the whole $\C\P^{n-1}$. So it realizes $\underline \E_1$ as a direct summand $\underline{\S}_1$ of $\underline Q_1$ by Corollary $2.5$ and Corollary $2.6$ in \cite{CS1}. This proves $(a)_1$ by letting $\S_1=\pi^* \underline \S_1$. 
\item[$(b)_1$] Since $\S_1$ is locally free away from $\Sigma$ and $\psi_1$ maps $\E_1$ isomorphically onto $\S_1$, we know in particular $\E_1$ must be locally free away from $\Sigma$, and $\psi_1$ is a vector bundle isomorphism away from $\Sigma$. By construction the bundle homomorphism $\psi_1$ then factors through the bundle homomorphism $\E_1\rightarrow Im(\E_1)\subset \E$. Hence on $B\setminus \Sigma$,  the map $\E_1\rightarrow\E$ must be an injective vector bundle homomorphism, and so $\E/\E_1$ is locally free. This implies $\Sing(\E/\E_1)\subset\Sigma$. Given any $z\notin \Sigma$,  choose a local orthonormal frame $\{e_{t}|1\leq t \leq \rk(\underline \E_1)\}$ for $\S_1$ near $z$. Then we can write $e_t=\sum_{l}a^t_l (z)\sigma^\infty_{1,l}(z)$ for each $t$, hence $\{e_t^{j_i}=\sum_{l} a^t_{l}\frac{ \sigma^{j_i}_{1,l}}{M^{j_i}_{1}}: 1\leq t \leq \rk(\S_1) \}$ is an approximately orthonormal frame of $(2^{-j_i})^*(\E_1)$ near $z$ which converges to $\{e_t\}$
smoothly. In particular, $\{\pi^{j_i}_{1}\}_i$ converges smoothly to $\pi^\infty_{1}$ given by $S_1\subset \E_\infty$. 
\item[$(c)_1$]  As in Lemma \ref{d0}, for any $\epsilon>0$, there exists a smooth Hermitian metric on $\underline H_\epsilon$ on $\underline{\E}$ so that 
$$
C_\epsilon |z|^{2\epsilon} H_\epsilon\leq H \leq C^{-1}_\epsilon |z|^{-2\epsilon}H_\epsilon
$$
for some constant $C_\epsilon>0$. Here $H_\epsilon=|z|^{2\mu(\underline \E)}\pi^* \underline H_\epsilon$.  Let $\pi^{\perp_{H_\epsilon}}_{1}\sigma_{2,l}$ denote the orthogonal projection of $\sigma_{2,l}$ to the orthogonal complement of $\E_{1}$ by using $H_\epsilon$. Then away from $\Sigma$ we have 
$$
|\pi^{\perp}_{1}\sigma_{2,l}|\leq  |\pi_1^{\perp_{H_\epsilon}}\sigma_{2,l}|\leq  C^{-1}_\epsilon |z|^{-2\epsilon} |\pi^{\perp_{H_\epsilon}}_{1}\sigma_{2,l}|_{H_\epsilon}.
$$
Here $|\cdot|$ denotes the norm defined by the unknown Hermitian-Yang-Mills metric $H$. Similarly 
$$C_\epsilon |z|^{2\epsilon} |\pi^{\perp_{H_\epsilon}}_{1}\sigma_{2,l}|_{H_\epsilon} \leq |\pi^{\perp}_{1}\sigma_{2,l}|.$$
 As a result, we have 
$$
d^{\epsilon}_{\E_{1}}(\sigma_{2,l})-\epsilon \leq d_{\E_{1}}(\sigma_{2,l}) \leq d^{\epsilon}_{\E_{1}}(\sigma_{2,l})+\epsilon, 
$$
where 
$$ d^{\epsilon}_{\E_{1}}(\sigma_{2,l})=\lim_{i\rightarrow\infty} \frac{\log \int_{(B_{2^{-j_i-1}}\setminus \overline{B}_{2^{-j_i-2}})\setminus  2^{-j_i}E^r_{j_i} }  |\pi^{\perp_{H_\epsilon}}_{1}\sigma_{2,l}|_{H_\epsilon}}{-2j_i\log 2}-n.$$
Since $H_\epsilon=|z|^{2\mu(\underline{\E})}\pi^* \underline H_\epsilon$, we have $\pi_{1}^{\perp_{H_\epsilon}}=\pi^* (\underline \pi^{\perp_{\underline H_\epsilon}}_{1})$. Using the fact that $\sigma_{2, l}=\pi^*\underline\sigma_{2,l}$, it is easy to see 
$$d^\epsilon_{\E_{1}}(\sigma_{2,l})=\mu(\underline\E)$$ 
 Then we have 
$$
\mu(\underline\E)-\epsilon\leq d^r_{\E_{1}}(\sigma_{2,l})\leq \mu(\underline \E)+\epsilon
$$
for any $\epsilon>0$. By letting $\epsilon\rightarrow 0$, we have $d^r_{\E_{1}}(\sigma_{2,l})=\mu(\underline \E)$. 
\end{itemize}

\textbf{Induction on $p$:}

\

 Suppose we have established the statements 
$$(a)_1, (b)_1, (c)_1,\cdots, (a)_{p-1}, (b)_{p-1},(c)_{p-1}.$$ 
\begin{itemize}
\item[$(a)_p$]  By $(c)_{p-1}$,  we have $d^r_{\E_{p-1}}(\sigma_{p, l})=\mu(\underline \E)$ for any $l$. By Corollary \ref{Cor2.11}, after passing to subsequence, $\{\frac{\sigma^{j_i}_{p,l}}{M^{j_i}_p}\}_{i}$ converges to homogeneous sections $\sigma^\infty_{p,l}$ of  $\E_\infty$ with degree $\mu(\underline \E)$ over $(B_{2^{-1}}\setminus \overline B_{2^{-2}}) \setminus \Sigma$ for each $l$ and at least one of them is  non-zero. By assumption, we have the following exact sequence of coherent sheaves on $B_{2^{-1}}\setminus \overline B_{2^{-2}}$, 
$$
0\rightarrow R_p \rightarrow \O^{\oplus n_p}\xrightarrow{\phi_p} \E_p/\E_{p-1}\rightarrow 0
$$
where $\phi_p(a_1, \cdots, a_{n_p})=\sum_l a_l\sigma_{p,l}$.  By $(b)_{p-1}$ we know $\E/\E_{p-1}$ is locally free away from $\Sigma$. For any $z\notin \Sigma \cup \Sing(\E_p/\E_{p-1})$, we define 
$$\phi^\infty_p: \O^{\oplus n_p} \rightarrow \E_\infty$$
by letting 
$$\phi^\infty_p(z)(a_1,\cdots, a_{n_p})=\lim_{i\rightarrow\infty} \frac{\sum_{l} a_{l}\sigma^{j_i}_{p, l}(z)}{M_p^{j_i}}=\sum_l a_l \sigma^{\infty}_{p,l}(z).$$ 
If $\phi_p(z)(\sum_{l}a_l \sigma_{p,l})=0$ in $\E_p/\E_{p-1}$, then by definition, we have $\sum_{l}a_l \sigma_{p,l}\in \E_{p-1}$, so $\sum_{l}a_l \sigma^{j_i}_{p,l}=0$.  Hence away from $\Sigma \cup \Sing(\E_p/\E_{p-1})$,  $\phi_p^{\infty}$ induces a nontrivial map 
$$\psi_p: \E_p/\E_{p-1}\rightarrow \E_\infty$$ 
which satisfies $\psi_p(\phi_p(\sigma_{p,l}(z)))=\sigma^\infty_{p,l}(z)$ for $z\notin \Sigma \cup\Sing(\E_p/\E_{p-1})$.  Let $Q_p$ be the maximal simple HYM cone summand containing the image of $\psi_p$. By $(b)_{p-1}$ and using the definition, we have 
$$
Q_p\subset (\E_{p-1}^\infty)^{\perp}
$$
away from $\Sigma$. Since $\sigma^{\infty}_{p,l}$ are all homogeneous sections of degree equal to $\mu(\underline \E)$, $\psi_p$ descends to  a nontrivial holomorphic map over $\C\P^{n-1} \setminus (\pi(\Sigma)\cup \text{Sing}(\underline\E/\underline\E_p))$ as 
$$\underline \psi_p: \underline \E_p /\underline \E_{p-1} \rightarrow \underline{Q}_p$$
where  $\mu(\underline \E_p/\underline \E_{p-1}) =\mu(\underline \E)=\mu(\underline Q_p)$. Then $\underline{\psi}_p$ extends to be a nontrivial holomorphic map defined over $\C\P^{n-1}$ and induces the following isomorphism 
$$\underline \psi_p: (\underline \E_p /\underline \E_{p-1})^{**} \rightarrow \underline{\S}_p, $$
where $\underline \S_p$ is a direct summand of $\underline Q_p$.
This proves $(a)_p$ by letting $\mathcal S_p=\pi^*\underline \S_p$. 
\item[$(b)_p$] By $(b)_{p-1}$, $\E/\E_{p-1}$ is  locally free away from $\Sigma$. Since $\E_{p}/\E_{p-1}$ is saturated in $\E/\E_{p-1}$, we know $\E_{p}/\E_{p-1}$ is reflexive away from $\Sigma$ by Proposition $5.22$ in \cite{Kobayashi}. Then away from $\Sigma$, $\psi_p$ is an isomorphism between $\E_p/\E_{p-1}$ and $\S_p$ which follows from the fact that reflexive sheaf is normal. Since $\S_p$ is locally free away from $\Sigma$, we know $\E_p/\E_{p-1}$ is also locally free away from $\Sigma$, and $\psi_p$ is a vector bundle isomorphism. As in the case $p=1$, since the map $\psi_p$ factors through the natural map $\E_p/\E_{p-1}\rightarrow \E/\E_{p-1}$, it follows that away from $\Sigma$, $\E_p/\E_{p-1}$ is a sub-bundle of $\E/\E_{p-1}$, and hence $\E_p$ is a sub-bundle of $\E$. This is equivalent to saying that $\E/\E_p$ is locally free away from $\Sigma$. For any $z\notin \Sigma$, we can choose $\{e'_t|1\leq t\leq \rk(\S_p)\}$ to be an orthonormal frame for $\S_p$ near $z$. Then we can write $e'_t=\sum_{l} a^t_{p,l}\sigma^\infty_{p,l}$ for each $t$ near $z$ and thus $\{\sum_{l} a^t_{p,l}\frac{\sigma^{j_i}_{p,l}}{M_p^{j_i}}|1\leq t \leq \rk(\S_p)\}$ is an approximately orthonormal frame for $(2^{-j_i})^{*}(\E_p)\cap ((2^{-j_i})^*(\E_{p-1}))^{\perp}$ near $z$ which smoothly converge to $\{e'_t: 1\leq t \leq \rk(\S_p)\}$. In particular, we have $\{\pi^{j_i}_{p}-\pi^{j_i}_{p-1}\}_i$ converges to $\pi^{\infty}_p $ given by $\S_p\subset \E_\infty$. Combining this with $(b)_{p-1}$, we have $\{\pi_p^{j_i}\}_i$ converges to the projection determined by $\oplus_{1\leq l\leq p} \S_l\subset \E_\infty$. This proves $(b)_p$.

\item[$(c)_p$] Finally $(c)_p$ follows line by line by replacing $\E_1$ with $\E_p$ in the proof $(c)_1$. So we have established $(a)_p, (b)_p, (c)_p$. This finishes the proof. 
\end{itemize}

\subsection{General Case}\label{unstable}
Now we assume $\underline \E$ is a general holomorphic vector bundle over $\C\P^{n-1}$. Let 
$$0=\underline \E_0\subset \underline \E_1\subset\cdots \underline\E_{m}=\underline \E$$ 
be the Harder-Narasimhan filtration of $\underline \E$, with $\mu_p=\mu(\underline\E_p/\underline\E_{p-1})$ strictly decreasing in $p$, and choose a filtration
$$\underline \E_{p-1}=\underline\E_{p, 0}\subset \underline \E_{p, 1}\subset \cdots \underline \E_{p, q_p}=\underline \E_{p}$$
so that 
$$0=\underline \E_{p, 1}/\E_{p-1}\subset \cdots \underline \E_{p, q_p}/\underline{\E}_{p-1}=\underline \E_{p}/\underline{\E}_{p-1}$$
is a Seshadri filtration of $\underline \E_p/\E_{p-1}$. By tensoring $\underline \E$ with  $\O(k)$ for $k$ large, we may assume the following for all $p$ and $q$, 
\begin{itemize}
\item $\underline\E_p$ and  $\underline\E_{p, q}$ are generated by its global sections;
\item we have a short exact sequence
$$0\rightarrow H^0(\C\P^{n-1}, \underline \E_{p, q-1})\rightarrow H^0(\C\P^{n-1}, \underline \E_{p, q})\rightarrow H^0(\C\P^{n-1}, \underline \E_{p, q}/\underline \E_{p, q-1})\rightarrow 0.$$
\end{itemize}

For $p=1, \cdots, m$, we  define 
$$HG_p:=\{s=\pi^*\underline s|\underline s\in H^0(\C\P^{n-1}, \underline \E_p)\}$$  and  
$$HG_{p,q}:=\{s=\pi^*\underline s|\underline s\in H^{0}(\C\P^{n-1}, \underline \E_{p,q})\}. $$
Then we have a filtration 
$$0 \subset HG_{1,1}\subset\cdots HG_{1,q_1}=HG_1\subset  \cdots \subset HG_m=HG. $$
In the following, we deonote $\E_p=\pi^*\underline{\E}_p$, $\E_{0,0}=0$ and $\E_{p,q}=\pi^*\underline{\E}_{p,q}$. Now we can repeat the proof in Section 3.1 for the semistable case here. The only difference in the proof is the calculation of the degree.

The following is taken from Proposition $3.22$ in \cite{CS1}.
\begin{prop}\label{constructedmetric}
For any $0<\epsilon<<1$, there exists a smooth Hermitian metric $H_{\epsilon}$ on $\E|_{B^*}$ satisfying the following
\begin{enumerate}[(i).]
\item $|F_{(H_{\epsilon}, \bp_\E)}|\in L^{1+\delta}(B^*)$ for some $\delta>0$;
\item $\limsup_{r\rightarrow 0} r^{1-2n}\int_{S^{2n-1}(r)}r^2|\Lambda_{\omega_0}F_{(H_\epsilon,\bp_\E)}| \leq\epsilon$;
\item $|z|^2|\Lambda_{\omega_0}F_{(H_{\epsilon},\bp_\E)}(z)|\leq \epsilon$ for $z\notin U$;
\item Outside $U$, we have 
$$
H_\epsilon =\sum_{p'} \pi^* \underline H_\epsilon(|z|^{\mu(\underline \E_{p'}/\underline \E_{p'-1} )}(\pi_{\epsilon, p'}-\pi_{\epsilon, p'-1})\cdot , |z|^{\mu(\underline \E_{p'}/\underline \E_{ p'-1} )}(\pi_{\epsilon, p'}-\pi_{\epsilon,p'-1})\cdot)  
$$
\end{enumerate}
Here $U=\pi^{-1}(\underline U) \cap B$ for some neighborhood of $\underline U$ of $Sing(Gr^{HNS}(\underline \E))$ and $\pi_{\epsilon,p'}$ denotes the orthogonal projection from $\E$ to $\E_{p'}$ defined by $\pi^*(\underline{H}_\epsilon)$.
\end{prop}

For any $\epsilon>0$, let $H_\epsilon$ be the metric above. Then we have the following (see also the proof of Proposition 3.17 in \cite{CS1})
\begin{cor}\label{cor}
$H\geq C |z|^\epsilon H_\epsilon$ for some constant $C$ outside $U$.
\end{cor}

\begin{proof}
By Lemma \ref{curvature}, since $F_{H_\epsilon}\in L^{1+\delta}(B^*)$ for some $\delta>0$ and $|F_{H}|\in L^2(B^*)$, we know 
$$\log^{+}\Tr H^{-1}H_\epsilon \in L^{\frac{n}{n-1}(1+\delta')}$$
where $\delta'=\min\{\delta, 1\}$. The conclusion now follows from applying Lemma \ref{pde} with $g=\log^{+} \Tr H^{-1}H_\epsilon$ and Item $(1)$ in Lemma \ref{curvature}.
\end{proof}

Fix any $r\in (0,r_0]$, where $r_0$ is given in Theorem \ref{Cut-Off}.  The following is an analogue of Proposition 3.17 in \cite{CS1}. Suppose $\lim_{i\rightarrow\infty}A_{j_i}=(A_\infty, \Sigma, \mu)$ is an analytic tangent cone and the rescaled projections $\{(2^{-j_i})^* \pi_{p,q}\}_i$ given by the orthogonal projection $\pi_{p,q}: \E\rightarrow \E_{p,q}$ converges to a projection map $\pi^{\infty}_{p,q}$ so that $\pi^\infty_{p,q}$ determines a direct HYM cone summand of $\E_\infty$. In particular this ensures $**$ in section \ref{convexity}  holds hence $d^r_{\E_{p, q}}$ is well-defined. Similar to the unstable case in \cite{CS1}, we can only get one-sided bound for the degree by using our analytic method. 

\begin{prop}\label{prop3.4}
Under the above assumption, we have for any $s\in HG_{p,q+1}\setminus HG_{p,q}$,  $$d^r_{\E_{p,q}}(s)\leq \mu(\underline\E_{p,q+1}/\underline\E_{p,q})$$
where we make the convention that when $q=q_p$, $(p, q+1)$ denotes $(p+1, 0)$.

\end{prop}

\begin{proof}
 For any $\epsilon>0$, let $H_\epsilon$ be the metric in Proposition \ref{constructedmetric}. By Corollary \ref{cor}, we know outside $U$, 
$$H\geq C |z|^\epsilon H_\epsilon$$ 
for some constant $C$. Furthermore, away from $U$, we have 
$$
H_\epsilon =\sum_{p'} \pi^* \underline H_\epsilon(|z|^{\mu(\underline \E_{p'}/\underline \E_{p'-1} )}(\pi_{ \epsilon
, p'}-\pi_{\epsilon, p'-1})\cdot , |z|^{\mu(\underline \E_{p'}/\underline \E_{ p'-1} )}(\pi_{\epsilon, p'}-\pi_{\epsilon,p'-1})\cdot)  
$$
where $\pi_{\epsilon,p'}$ is the pointwise orthogonal projection from $\E$ to $\E_{p'}$ with respect to the metric $\pi^*\underline H_\epsilon$. Similar to the proof of $(c)_1$ in the semistable case, we have
\begin{equation}\label{eqn3.1}
d^r_{\E_{p,q}}(s) \leq \lim_{i\rightarrow\infty}\frac{ 
\log (2^{2j_i n}\int_{(B_{2^{-j_i-1}}\setminus \overline{B_{2^{-j_i-2}}})\setminus (U\cup  2^{-j_i}E^r_{j_i})} |z|^\epsilon |(\pi^{\epsilon}_{p,q})^{\perp}s|^2_{H_\epsilon})}{-2j_i\log 2}
\end{equation}
where $\pi^\epsilon_{p,q}$ denotes the pointwise projection given by $\E_{p,q} \subset \E$ with respect to the metric $H_\epsilon$.
However, over $(B_{2^{-j_i-1}}\setminus \overline{B_{2^{-j_i-2}}})\setminus (U\cup  2^{-j_i}E^r_{j_i})$, we have 
$$(\pi_{p,q}^{\epsilon})^\perp s=(\pi_{\epsilon, (p,q) })^{\perp}s
$$
where $\pi_{\epsilon, (p,q) }$ denotes the pointwise projection given by $\E_{p,q} \subset \E$ with respect to the metric $\pi^* \underline H_\epsilon$. Then we have 
$$
|(\pi_{p,q}^{\epsilon})^{\perp}s|^2_{H_\epsilon}=|z|^{2\mu(\underline\E_{p,q+1}/\underline\E_{p,q})} | (\pi_{\epsilon, (p,q) })^{\perp}s|^2_{\pi^*\underline H_\epsilon}. 
$$
By plugging this into Equation (\ref{eqn3.1}), we have 
$$
\begin{aligned}
d^r_{\E_{p,q}}(s) 
&\leq \lim_{i\rightarrow\infty}\frac{ 
\log (2^{2j_i n}\int_{(B_{2^{-j_i-1}}\setminus \overline{B_{2^{-j_i-2}}})\setminus (U\cup  2^{-j_i}E^r_{j_i})} |z|^{2\mu(\underline\E_{p,q+1}/\underline\E_{p,q})+\epsilon} | (\pi_{\epsilon, (p,q) })^{\perp}s|^2_{\pi^*\underline H_\epsilon}}{-2j_i\log 2}\\
&=\mu(\underline\E_{p,q+1}/\underline\E_{p,q})+\frac{\epsilon}{2}.
\end{aligned}
$$
By letting $\epsilon \rightarrow 0$, we have 
$$d^r_{\E_{p,q}}(s)\leq \mu(\underline\E_{p,q+1}/\underline\E_{p,q}). $$
This finishes the proof.
\end{proof}

\begin{rmk}
\begin{itemize}
\item It will follow from the proof of Theorem \ref{main} that the equality holds in Proposition \ref{prop3.4}.
\item When $(p,q)=(0,0)$, the assumption holds trivially and thus for any nonzero $s\in HG_{1,1}$, $d_0^r(s)\leq \mu(\underline \E_{1,1})$.
\end{itemize}
\end{rmk}

Given this, we can finish the proof of Theorem \ref{main} by repeating what we did in the semistable case and replacing the Harder-Narasimhan filtration with a Harder-Narasimhan-Seshadri filtration. Indeed, Theorem \ref{main} is a direct consequence if we can prove the following statements by doing induction on $(p,q)$. 

\begin{itemize}
\item[$(a)_{p,q}$.] There is a simple HYM cone direct summand $\S_{p,q}$ of $\E_\infty$ which is isomorphic to $(\E_{p,q}/\E_{p,q-1})^{**}$ so that $\S_{p,q} \perp \S_{p',q'}$ for any $p'\leq p$ and $q'\leq q'$ with $p'+q'<p+q$. (We take $\S_{0,0}=0$ here);
\item[$(b)_{p,q}$.] $\Sing(\E_{p,q}/\E_{p,q-1})\cup \Sing(\E/\E_{p,q})\subset \Sigma$, and   over $(B_{2^{-2}}\setminus B_{2^{-1}})\setminus \Sigma$, $\{\pi^{j_i}_{p,q}\}_{i}$ converges locally smoothly to the limit projection $\pi^\infty_{p,q}: \E_\infty \rightarrow \E_\infty$ given by $\oplus_{k\leq p} \S_k \subset \E_\infty$.  Here $\pi^{j_i}_{p,q}=(2^{-j_i})^*\pi_{p,q}$. In particular, the hypothesis $**$ in section \ref{convexity}  holds with $\E_{p,q}$ which enable us to run the convexity result.
We also denote by $(\E^{\infty}_{p,q})^\perp$  the HYM cone direct summand of $\E_\infty$ so that 
$$\E_\infty=\oplus_{k\leq p} \S_k \oplus (\E^{\infty}_{p,q})^\perp.$$ 
\item[$(c)_{p,q}$.] $d_{\E_{p,q}}(\sigma)=\mu(\underline \E_p)$ for any $\sigma \in HG_{p,q+1}\setminus HG_{p,q}$. 
\end{itemize}
The only difference from the semistable case is that we only have one-sided bound of the degree. However the argument itself will force the inequality in Proposition \ref{prop3.4} to hold and so the argument from the semistable case applies exactly line by line in the general case.

\

\textbf{Base case $(1,1)$:} 

\

\begin{itemize}
\item[$(a)_{1,1}$]  We fix a basis for $HG_{1,1}$ as $\{\sigma_{(1,1), l}: 1\leq l \leq n_{1,1}\}$ where $n_{1,1}=\dim_{\C} HG_{1,1}$. Let $s_{l}^{j}=(2^{-j})^* \sigma_{(1,1),l}$. Define $M^{i}_{1,1}=max_{l} \|s_l^j\|^r_j$. By Proposition \ref{prop3.4} and Proposition \ref{prop2.23}, after passing to further subsequence we may assume $\{\frac{\sigma^{j_i}_{1,l}}{M_{1,1}^{j_i}}\}_i$ converges to a holomorphic homogeneous section of degree less or equal to $\mu(\underline \E_1)$ away from $\Sigma$ for any $1\leq l\leq n_1$ (we will see the degree will be all equal to $\mu(\underline \E_1)$) and at least one of the limits is non-zero.  By assumption we have the following exact sequence of coherent sheaves 
$$
0\rightarrow R_{1,1} \rightarrow \O^{\oplus n_1} \xrightarrow{\phi_{1,1}} \E_{1,1}\rightarrow 0
$$
where 
$$\phi_{1,1}(z)(a_1,\cdots, a_{n_1})=\sum_{ l=1}^{ n_1} a_{l} \sigma_{(1,1),l}(z).$$
 Away from $\Sing(\E/\E_{(1,1)})$, $\E_{1,1}$ can be viewed as a vector sub-bundle of $\E$. For $z\notin \Sigma\cup \text{Sing}(\E/\E_{1,1})$, we define a vector bundle homomorphism
$$
\phi_{1,1}^\infty: \O^{\oplus n_1}\rightarrow \E_\infty
$$
by 
$$\phi_{1,1}^{\infty}(z)(a_1,\cdots, a_{n_1})=\lim_{i\rightarrow\infty}{(M^{j_i}_{1,1})}^{-1}{\sum_{l=1}^{n_1} a_{l} \sigma^{j_1}_{1,l}(z)}.$$
 If  $(a_1,\cdots, a_{n_1})$ is in the fiber of $(R_{1,1})_z$, then by definition, we have $\sum_{l=1}^{n_1} a_{l} \sigma_{(1,1),l}(z)=0$, hence $\sum_{l=1}^{n_1} a_{l} \sigma_{(1,1),l}^{j_i}(z)=0$, which implies 
 $$\phi^{\infty}_{1,1}(z)(a_1,\cdots, a_{n_1})=0$$ and $\phi_{1,1}^\infty$ descends to a homomorphism away from $\Sigma$
$$
\psi_{1,1}: \E_1\simeq\O^{n_1}/R_{1,1}\rightarrow \E_\infty
$$
which satisfies $\psi_{1,1}(z)(\sigma_{(1,1),l}(z))=\sigma^{\infty}_{(1,1),l}(z)$. Furthermore, it is easy to notice that 
all the nonzero sections among $\{\sigma^{\infty}_{(1,1),l}\}_l$ have the same degree which is equal to the minimum of $d^r_0(\sigma_{(1,1), l})$ for all $l$. We can let $Q_{1}$ be the maximal simple HYM cone direct summand of $\E_\infty$ which contains the image of $\psi_{1,1}$. Note $Q_{1}$ is locally free away from $\Sigma$. By Lemma \ref{descend}, $\psi_{1,1}$ descends to be a nontrivial map defined over $\C\P^{n-1}\setminus (\pi(\Sigma)\cup \text{Sing}(\underline\E/\underline\E_1))$
$$
\underline \psi_{1,1}: \underline \E_{1,1} \rightarrow \underline Q_1,
$$
where $\mu(\underline \E_1)\geq\mu(\underline Q_1)$. Notice this is the only place we need to change compared to $(a)_1$ in the semistable case. By Lemma $3$ in \cite{Shiffman}, $\underline \psi_{1,1}$ extends to a sheaf homomorphism over the whole $\C\P^{n-1}$. So it has to be injective for otherwise the kernel of this map will violate the stability of $\underline{\E}_{1,1}$. In particular, $\mu(\underline \E_{1,1})=\mu(\underline Q_1)$ and $d^r_0(\sigma_{(1,1),l})=\mu(\underline \E_{1,1})$ for any $s\in HG_{1,1}$. Again, $\underline \psi_{1,1}$ realizes $\underline \E_{1,1}$ as a direct summand $\underline{\S}_{1,1}$ of $\underline Q_1$.  This proves $(a)_{1,1}$.
\item[$(b)_{1,1}$] Exactly the same argument as $(b)_1$ in the semistable case;
\item[$(c)_{1,1}$] Exactly the same argument as $(c)_1$ in the semistable case; 
\end{itemize}

\textbf{Induction on $(p,q)$:} 

\

Again, this is exactly the same as the semistable case except that we only have one-sided bound of the degree given by Proposition \ref{prop3.4}. However, the argument as in $(a)_{1,1}$ shows the equality holds in Proposition \ref{prop3.4}. The arguments in the semistable case can be then applied to finish the proof.

\section{Uniqueness of bubbling set with multiplicities}\label{UniquenessOfBubblingSet}
\subsection{Chern-Simons transgression}
In this section, we will collect some well-known results about the Chern-Simons transgression. Fix $\Delta$ to be smoothly isomorphic to $\{z\in \C^2: |z|\leq 1\}$ and let $E$ be a complex vector bundle of rank $m\geq 2$ over $\Delta$ with a preferred smooth trivialization over $\p\Delta$ (indeed $E$ is always abstractly trivial). Then any connection $A$ defined on $E|_{\p \Delta}$ can be viewed as a smooth one form and the Chern-Simons form is defined as
$$
CS(A)=\Tr(dA \wedge
 A+ \frac{2}{3} A \wedge A \wedge A).
$$
Given two such connections $A$ and $B$, we also define the relative Chern-Simons transgression form as
$$
CS(A, B):=\Tr(d_{B} a \wedge a +\frac{2}{3} a\wedge a \wedge a + 2 a\wedge F_{B}).
$$
Here $a=A-B$. Note $CS(A)=CS(A,0)$. Given a smooth isomorphism $g: E|_{\p\Delta}\rightarrow E|_{\p \Delta}$, we define the (complex) gauge transform of a connection $A$ on $E|_{\p \Delta}$ as 
$$
g\cdot A = g A g^{-1}-d g \cdot g^{-1}.
$$
\begin{lem}\label{CS}
The following holds
\begin{itemize}
\item[(a).] if $A$ and the preferred trivialization extend to a smooth connection of $E$ over the whole $\Delta$, then 
$$
\int_{\p\Delta}CS(A)=\int_{\Delta} \Tr(F_{A}\wedge F_{A});
$$
\item[(b).] $\int_{\p \Delta} CS(A, B)=\int_{\p\Delta}CS(A)-\int_{\p\Delta}CS(B)$;
\item[(c).] For any $g$ as above, $CS(g\cdot A, g\cdot B)=CS(A,B)$. In particular the relative Chern-Simons transgression does not depend on the choice of the common trivialization of $E|_{\p \Delta}$;
\item[(d).] $\deg(g):=\int_{\p\Delta}CS(g\cdot A, A)\in 8\pi^2\Z$ is independent of $A$ and it only depends on the isotopy class of $g$. Moroever, $\deg(g_1 g_2)=\deg(g_1)+\deg(g_2)$;
\item[(e).] If $g$ extends to be an isomorphism of $E$ over $\Delta$, then  $\deg(g)=0$.
\end{itemize}
\end{lem}
\begin{proof}
$(a)$ follows from the fact that 
$$d(CS(A))=Tr(F_A\wedge F_A).$$
$(b)$ also follows from a direct calculation 
$$
CS(A)=CS(B)+CS(A,B)+d \Tr(a\wedge B).
$$
For $(c)$ we write $g\cdot A-g\cdot B=g(A-B)g^{-1}$. Denote $a=A-B$ and $g\cdot a = ga g^{-1}$. Then we have 
$$
\begin{aligned}
CS(g\cdot A, g\cdot B)
&=\Tr(d_{g\cdot B}g\cdot a \wedge g\cdot a +\frac{2}{3}g\cdot a\wedge g\cdot  a \wedge g\cdot a + 2 g\cdot a\wedge F_{g\cdot B})\\
&=Tr(g(d_Ba \wedge a+\frac{2}{3}a\wedge a \wedge a+2a\wedge F_{B}) g^{-1})\\
&=CS(A,B).
\end{aligned}
$$
For $(d)$, by $(b)$ and $(c)$, we have
$$
\begin{aligned}
\int_{\p \Delta} CS(A, g\cdot A)-CS(B, g\cdot B)
&=\int_{\p\Delta} CS(A)-CS(g\cdot A)-CS(B)+CS(g\cdot B)\\
&=\int_{\p \Delta}CS(A,B)-CS(g\cdot A, g\cdot B)\\
&=\int_{\p \Delta} CS(A,B)-CS(A,B)\\
&=0.
\end{aligned}.
$$ 
To see $\deg(g)\in 8\pi^2\Z$, we take the trivial connection $A_0$ on $E$ over $\Delta$, so $CS(A_0)=0$. Then we take another copy of $A_0$ and glue these two together along $\p\Delta$ using $g$ to form a connection $A_1$ on a bundle over $S^4$. Then we have 
$$
\deg(g)=CS(g\cdot A_0, A_0)=CS(g\cdot A_0)-CS(A_0)=\int_{S^4} Tr(F_{A_1}\wedge F_{A_1})\in 8\pi^2\Z. $$
Also we have 
$$
\begin{aligned}
\deg(g_1g_2)
&=\int_{\p \Delta} CS(A, g_1 g_2 A)\\
&=\int_{\p \Delta} CS(A, g_1 A)+ \int_{\p \Delta} CS(g_1 A, g_1 g_2 A)\\
&=\int_{\p \Delta} CS(A, g_1 A)+ \int_{\p \Delta} CS(A, g_2 A)\\
&=\deg(g_1)+\deg(g_2).
\end{aligned}.
$$
$(e)$ follows by using the same gluing argument. 
\end{proof}

\subsection{Proof of Theorem \ref{Bubbling set}}
In this section, we will prove Theorem \ref{Bubbling set}. By Theorem \ref{main},  given any analytic tangent cone $(\E_\infty, A_\infty, \Sigma, \mu)$,we know that the connection $A_\infty$ is given by the admissible Hermitian-Yang-Mills connection on $(Gr^{HNS}(\underline\E))^{**}$. More specifically,  write 
$$(Gr^{HNS}(\underline\E))^{**}=\oplus_l \underline Q_l, $$ 
where each $\underline Q_l$ is a stable reflexive sheaf on $\C\P^{n-1}$. Let $\underline S$ denote the set where $(Gr^{HNS}(\underline \E))^{**}$ is not locally free and $\mu_l$ denote the slope of $\underline Q_l$. Then Theorem \ref{main} tells us that away from $\pi^{-1}(\underline S)$,  we have
$$(\E_\infty, A_\infty, H_\infty)=(\pi^*(Gr^{HNS}(\underline \E))^{**},\bigoplus_l (\pi^* \underline A_l + \mu_l \partial \log |z|^2 \Id_{\pi^*\underline Q_l}), \bigoplus_l |z|^{2\mu_l}\pi^*\underline H_l)
$$  where $(\underline A_l, \underline H_l)$ is the (unique) admissible Hermitian-Yang-Mills connection over $\underline Q_l$. In particular 
$$\text{Sing}(A_\infty)=\pi^{-1}(\underline S).$$ 
In the following, we denote 
$$\underline A := \oplus_l \underline A_l$$ and 
$$a_\infty:=\oplus_l \mu_l \p \log|z|^2 \Id_{\pi^* \underline Q_l}.$$

Let $(A_\infty, \Sigma, \mu)$ be an analytic tangent cone associated to a subsequence $\{j_i\}\subset\{i\}$. Let $\underline H'$ be a fixed smooth Hermitian metric on $\underline \E$ and let $\underline A'$ be the Chern connection of $(\underline H', \bp_\E)$. Denote $H'=\pi^*\underline H'$.  Following the convention in Section \ref{Tangent Cone}, there exits a unitary isomorphism $P$ outside $\Sigma$ 
$$
P: (\E, H')\rightarrow (\pi^*(Gr^{HNS}(\underline \E))^{**}, H_\infty)
$$
and a sequence of unitary isomorphism $\{g_i\}$ of $(\E, H')$ so that $\{g_i \cdot A_{j_i}\}_i$ converge to $P^{*} A_\infty$ smoothly outside $\Sigma$. Here $A_{j_i}$ denotes the Chern connection associated to $(H', f_i\circ\bp_{\pi^*\underline \E}\circ f_{i}^{-1})$ where $f_i=(H'^{-1} (2^{-j_i})^* H)^{\frac{1}{2}}$. We fix a Harder-Narasihman-Seshadri filtration for $\underline \E$ as 
$$
0\subset \underline \E_1 \subset \cdots \underline\E_m=\underline \E. 
$$
Let $\underline Q'_l$ be the orthogonal complement of $\underline \E_{l-1}$ in $\underline \E_{l}$ with respect to $\underline H'$. By doing  orthogonal projection, we can identify $\underline \E$ smoothly with $\oplus_l \underline Q_l'$ away from $\text{Sing}(Gr^{HNS}(\underline \E))\subset \Sigma$ 
\begin{equation}\label{1}
\underline \rho: \underline \E\rightarrow Gr^{HNS}(\underline \E).
\end{equation}
We also denote $\rho=\pi^*\underline \rho$. Let 
\begin{equation}\label{2}
\underline\iota: Gr^{HNS}(\underline \E) \hookrightarrow (Gr^{HNS}(\underline \E))^{**} 
\end{equation}
be the natural inclusion map.

Now we will follow the discussion in \cite{SW}. 
Let $\underline \Sigma^{alg}$ denote the proper analytic subvariety in $\C\P^{n-1}$ where $Gr^{HNS}(\underline \E)$ is not locally free. 
Define 
$$
\underline\T:=(Gr^{HNS}(\underline \E))^{**}/Gr^{HNS}(\underline \E)
$$
which is a torsion sheaf over $\C\P^{n-1}$. Then we have 
$$
\underline\Sigma^{alg}=\text{supp}(\underline \T) \cup \underline S.
$$
By  Proposition $2.3$ in \cite{SW}, on the complement of $\underline S$, $\underline \Sigma^{alg}$ has pure complex codimension $2$. Then we define $\underline \Sigma^{alg}_b$ as the union of irreducible codimension $2$ components in $\underline\Sigma^{alg}$. For each irreducible component $\underline \Sigma_k$ of $\underline \Sigma_b^{alg}$, we can associate an algebraic multiplicity $m^{alg}_k$ to $\underline\Sigma_k$ by letting 
$$m_k^{alg}:=h^{0}(\underline\Delta, \underline \T |_{\underline \Delta} )$$ 
where $\underline\Delta$ is a holomorphic transverse slice at a generic point of $\underline \Sigma_k$.  We write
$$\Sigma^{alg}_b=\sum_k m_k^{alg} \Sigma_k$$
where $\Sigma_k=\pi^{-1}(\underline \Sigma_k)$.

Given an irreducible component $\Sigma_k$ of $\Sigma_b^{an}\cup \Sigma_b^{alg}$, it has been shown how  to calculate the algebraic multiplicity in \cite{SW}.  More specifically, choose a class $[\underline \Delta]$ in $H_4(\C\P^{n-1}, \mathbb{Q})$ whose intersection product with $\underline \Sigma_k$ is nonzero and $[\underline \Delta]$ can be represented as a codimension $2$ subvariety $\underline \Delta$ of $\C\P^{n-1}$  which intersects $\underline \Sigma_k$ transversally and positively at points $\{\underline z_1,\cdots, \underline z_N\}$. Now we have the following (see Equation $(4.5)$ in \cite{SW}), 

\begin{equation}\label{eqn4.1}
\begin{aligned}
N m^{alg}_k=([\underline \Delta]\cdot [\underline\Sigma_k])m^{alg}_k
=&\int_{\cup^{N}_{l=1}\underline\Delta \cap \underline B_{\sigma}(z_l)}\frac{1}{8\pi^2}\{\tr(F_{\underline A'}\wedge F_{\underline A'})-\tr(F_{\underline\tau^*\underline A} \wedge F_{\underline\tau^*\underline A})\}\\&
-\int_{\cup^{N}_{l=1} \underline\Delta \cap \p (\underline B_{\sigma}(z_l) )}\frac{1}{8\pi^2}CS(\underline A', \underline\tau^{*}\underline A)
\end{aligned}
\end{equation}
Here $\underline \tau=\underline \iota\circ \underline \rho$, where $\underline\iota$ and $\underline\rho$ is defined in Equation \ref{1} and \ref{2}. In \cite{SW}, the result is only stated for an irreducible component of $\Sigma_b^{alg}$ but the calculation obviously holds for any codimension $2$ subvariety (indeed, if $\Sigma_k$ is not a component of $\Sigma^{alg}_b$, then $m^{alg}_k=0$). By choosing $\sigma$ small, for each $\underline\Delta_l$, we can choose a holomorphic lifting $\Delta_l$ in $B_{2^{-1}}\setminus \overline {B_{2^{-2}}}$. Then Equation (\ref{eqn4.1}) can be rewritten as 
\begin{equation}\label{eqn4.2}
\begin{aligned}
N m^{alg}_k=&\int_{\cup^{N}_{l=1}\Delta_l }\frac{1}{8\pi^2}\{\tr(F_{\pi^*\underline A'}\wedge F_{\pi^*\underline A'})-\tr(F_{\pi^*\underline A} \wedge F_{\pi^* \underline A})\}\\&
-\int_{\cup^{N}_{l=1} \p \Delta_l}\frac{1}{8\pi^2}CS(\pi^*\underline A', \tau^* (\pi^*\underline A)).
\end{aligned}
\end{equation}
where $\tau=\pi^*\underline\tau$.

\begin{cor}\label{Cor4.2}
$$Nm_k^{alg}=Nm_k^{an}-\lim_i\int_{\cup^{N}_{l=1} \p \Delta_l}\frac{1}{8\pi^2}CS(A_{j_i}, \tau^* A_\infty).$$
\end{cor}
\begin{proof}
By Lemma \ref{AnalyticMultiplicity} and Equation (\ref{eqn4.2}), using Lemma \ref{CS} we get
$$
\begin{aligned}
Nm_k^{an}
&=\lim_{i\rightarrow \infty}\int_{\cup^{N}_{l=1}\Delta_l }\frac{1}{8\pi^2}\{\tr(F_{A_{j_i}}\wedge F_{A_{j_i}})-\tr(F_{A_\infty} \wedge F_{A_\infty})\}\\
&=Nm^{alg}_{k}-\lim_{i\rightarrow \infty}\frac{1}{8\pi^2}\int_{\cup_{l=1}^N \p \Delta_l}CS(\pi^*\underline A', \tau^* (\pi^*\underline A) )+CS(A_{j_i}, \pi^*\underline A')+CS( \pi^*\underline A, A_\infty )\\
&=Nm^{alg}_{k}-\lim_{i\rightarrow \infty}\frac{1}{8\pi^2}\int_{\cup_{l=1}^N \p \Delta_l} CS(A_{j_i}, \tau^*(\pi^*\underline A))+CS(\tau^*( \pi^*\underline A),\tau^* A_\infty)\\
&=Nm^{alg}_k-\lim_{i\rightarrow\infty}\frac{1}{8\pi^2} \int_{U_{l=1}^N \p \Delta_l }CS(A_{j_i},\tau^*A_\infty).
\end{aligned}
$$
\end{proof}

Now Theorem \ref{Bubbling set} follows from the following Proposition combined with Corollary \ref{Cor4.2}.

\begin{prop}\label{Prop4.3} For all $1\leq l\leq N$, we have
$$\lim_{i\rightarrow\infty}\int_{\p\Delta_l}CS(A_{j_i}, \tau^*A_\infty)=0.$$
\end{prop}

\begin{proof}
First, we have 
\begin{align*}
\int_{\p\Delta_l}CS(A_{j_i}, \tau^*A_\infty)=\int_{\p\Delta_l}CS(g_i \cdot A_{j_i}, g_i\cdot (P^{-1}\tau)^* P^{*}A_\infty)\\=\int_{\p\Delta_l}CS(g_i \cdot A_{j_i}, g_i\cdot (P^{-1}\tau)^{-1}\cdot P^{*}A_\infty) 
\end{align*}
We claim for $i$ large,  on $\E|_{\p\Delta_l}$, we have
$$\deg(P^{-1}\tau)=\deg(g_i). $$
Given this claim,  by Lemma \ref{CS}, we have 
$$
\int_{\p\Delta_l}CS(A_{j_i}, \tau^*A_\infty)=\int_{\p\Delta_l} CS(g_i\cdot A_{j_i}, P^*A_\infty)
$$
which goes to $0$ since $\{g_i \cdot A_{j_i}\}_i$ converge to $P^*A_\infty$ smoothly away from $\Sigma$.

Now we prove the Claim. The key point is that in our proof of Theorem \ref{main} (see Section \ref{semistable}), the homogeneous map $\psi_l$ we constructed to identify 
$\pi^*(\underline\E_l/\underline\E_{l-1})^{**}$ with $\pi^*\underline Q_l$ is given by (away from $\Sigma$)
$$
\psi_l=P \lim_{i\rightarrow\infty} \frac{(\pi^i_{l}-\pi^i_{l-1})(g_i f_i(\pi_l-\pi_{l-1}))}{a_i}
$$
Here $\pi_l^i$ denotes the orthogonal projection from $\E$ to $(g_if_i)(\pi^*\underline\E_l)$ with respect to the metric $H'$, $\pi_l$ denote the orthogonal projection from $\E$ to $\pi^*\underline \E_l$ with respect to $H'$ and $a_i$ is suitable normalizing constant (see the proof in Section \ref{semistable}). Since the map between $\underline \E_l /\underline \E_{l-1}$ and $(\underline\E_l/\underline\E_{l-1})^{**}$ which induces an isomorphism of $(\E_l/\E_{l-1})^{**}$ is unique up to rescaling, we can assume $\tau=\bigoplus_l \psi_l$ by re-normalizing $a_i$. Write 
$$h_i:=\frac{(\pi^i_{l}-\pi^i_{l-1})(g_i f_i(\pi_l-\pi_{l-1}))}{a_i}, $$
then it is easy to see that $h_i$ is smoothly homotopic to $g_if_i$ away from $\Sigma$. Indeed, consider for $t\in[0, 1]$ the family
$$F_t= \frac{(\pi^i_{l}-\pi^i_{l-1})(g_i f_i(\pi_l-\pi_{l-1}))+t\sum_{l_1, l_2<l}(\pi^i_{l_1}-\pi^i_{l_1-1})(g_i f_i(\pi_{l_2}-\pi_{{l_2}-1}))}{(1-t)ta_i+t}$$
is a family of complex gauge transformations satisfying $F_0=h_i$ and $F_1=g_if_i$. Now since $f_i$ is defined over $B\setminus \{0\}$, we know by Lemma \ref{CS} that $\deg(f_i)=0$ on $\E|_{\p\Delta_l}$. So for $i$ large, we have 
$$
\begin{aligned}
\deg(P^{-1}\tau)
=&\deg(\lim_{i\rightarrow\infty} g_i f_i)\\
=& \lim_{i\rightarrow\infty} \deg g_i.
\end{aligned}
$$
This finishes the proof of the claim.
\end{proof}

\section{Examples}
In this section, we will prove Corollary \ref{Example}. We first state a lemma to construct reflexive sheaves in general. Suppose $\{f_1,\cdots f_k\}$ is a regular sequence of holomorphic function over an open subset $U\subset \C^n$ i.e. 
$$\text{Codim}_{\mathbb{C}}(\text{Zero}(f_1,\cdots f_k))=n-k.$$ 
Denote $u:=(f_1,\cdots f_k)\in \O^{\oplus k}$. Consider the coherent sheaf $\E$ given by the following exact sequence
$$
0\rightarrow \O\xrightarrow{u} \O^{\oplus k} \rightarrow \E\rightarrow 0.
$$ 
\begin{lem}\label{lem5.1}
$\E$ is a reflexive sheaf over $U$ for $k\geq 3$.
\end{lem}

\begin{proof}
Indeed, since $\text{Codim}_{\mathbb{C}}(\text{Zero}(f_1,\cdots f_k))=n-k,$ by Lemma on Page $688$ in \cite{GH}, the following Koszul complex given by $u$ is exact over $U$
$$
0\rightarrow \O \xrightarrow{\wedge u} \O^{\oplus k} \xrightarrow{\wedge u} \wedge^2 \O^{\oplus k} \xrightarrow{\wedge u}\wedge^3 \O^{\oplus k} \cdots \xrightarrow{\wedge u} \I \rightarrow 0
$$
where $\I$ is the ideal sheaf generated by $\{f_1,\cdots f_k\}$. By exactness of the above sequence and the definition of $\E$, we have the following exact sequence
$$
0\rightarrow  \E \rightarrow \wedge^2\O^{\oplus k} \rightarrow \wedge^3 \O^{\oplus k}
$$
which implies $\E$ is reflexive by Proposition $5.22$ in \cite{Kobayashi}.
\end{proof}

Now we discuss a class of \emph{local} examples. Over $\C\P^2$ we denote by $\underline\E_k$ the locally free rank 2 sheaf defined by the exact sequence 
$$0\rightarrow \O_{\C\P^2}\xrightarrow{f_k} \O_{\C\P^2}(1)\oplus \O_{\C\P^2}(1)\oplus \O_{\C\P^2}(k) \rightarrow\underline \E_k\rightarrow 0$$
where $f_k=(z_1, z_2, z_3^k)$.  Let $\E_k=\psi_*\pi^*\underline\E_k$. 

It is easy to see $c_1(\underline\E_k)=k+2$. By using the criteria given by Lemma $1.2.5$ in \cite{CMH}, we can easily get the following
\begin{itemize}
\item If $k=1$, then $\underline\E_k$ is stable; 
\item If $k=2$, then $\underline \E_k$ is semistable, and a Seshadri filtration is given by 
$$0\rightarrow \O_{\C\P^2}(k)\rightarrow \underline \E_k\rightarrow \I_{[0:0:1]}\otimes \O_{\C\P^2}(2)\rightarrow 0, $$
where $\I_{[0:0:1]}$ is the ideal sheaf of the point $[0:0:1]$, and the first map is induced by the inclusion $\O_{\C\P^2}(k)\hookrightarrow \O_{\C\P^2}(1)\oplus \O_{\C\P^2}(1)\oplus \O_{\C\P^2}(k)$.
\item If $k>2$, then $\underline \E_k$ is unstable and the Harder-Narasimhan filtration is given by 
$$0\rightarrow \O_{\C\P^2}(k)\rightarrow \underline \E_k\rightarrow \I_{[0:0:1]}\otimes \O_{\C\P^2}(2)\rightarrow 0.$$
\end{itemize}

When $k\geq 2$, we have  
$$Gr^{HNS}(\underline\E_k)=\O_{\C\P^2}(k)\oplus (\I_{[0:0:1]}\otimes \O_{\C\P^2}(2)), $$
so
$$\psi_*\pi^*(Gr^{HNS}(\underline\E_k))^{**}=\O_{\C^3}^{\oplus 2}, $$
and the algebraic bubbling set 
$$\Sigma_b^{alg}=\{0\}\times \C_{z_3}\subset\C^3$$
with multiplicity $1$.

Now suppose $A$ is an admissible Hermitian-Yang-Mills connection on $\E_k|_{B}(k\geq 2)$, and let $(A_\infty, \Sigma, \mu)$ be the unique analytic tangent cone of $A$ at $0$, then by Theorem \ref{main} and \ref{Bubbling set} we know $A_\infty$ is the trivial flat connection on $\O_{\C^3}^{\oplus 2}$, and the bubbling set is $\Sigma_b$, $\mu=[\Sigma_b]$. In particular, the analytic tangent cones, as defined in this paper, are all the same for all $k\geq 2$. 

\begin{rmk}
It is interesting to understand how to interpret the integer $k$ here in terms of the admissible Hermitian-Yang-Mills connections on $\E_k$.
\end{rmk} 
Finally, to prove Corollary \ref{Example}, we consider a global example. On $\C\P^{3}$,  we let $\E$ be given as follows
$$
0\rightarrow \O_{\C\P^3} \xrightarrow{\sigma} \O_{\C\P^3}(2)\oplus \O_{\C\P^3}(1)\oplus \O_{\C\P^3}(2) \rightarrow \E\rightarrow 0.
$$
where $\sigma=(z^2_1, z_2, z_3z_4)$. 
\begin{lem}
$\E$ is a rank $2$ stable reflexive sheaf with singular set given by $\{[0,0,0,1], [0,0,1,0]\}$.
\end{lem}
\begin{proof}
It is obvious that $\E$ is locally free away from $\{[0,0,0,1], [0,0,1,0]\}$. Near $[0,0,0,1]$, since $[0,0,0,1]$ is an isolated common zero of $\{z_1^2, z_2, z_3z_4\}$, by Lemma \ref{lem5.1}, we know $\E$ is reflexive near $[0,0,0,1]$. Similarly, $\E$ is also reflexive at $[0,0,1,0]$. Since $H^{0}(\C\P^{n-1}, \E\otimes \O_{\C\P^3}(-3))=0$, $\E$ is stable.  
\end{proof}

By Theorem $2$ in \cite{BS}, on $\C\P^3$ endowed with the standard Fubini-Study metric,  there exists an admissible Hermitian-Yang-Mills connection $A$ on $\E$ with  singularities at $p_0=[0,0,0,1]$ and $p_1=[0,0,1,0]$. We will apply our Theorem \ref{main} and \ref{Bubbling set} to study the local behavior of $A$ near $p_0$ and $p_1$.  Locally around $p_0$, $\E$ is isomorphic to $\E_2$ (as defined as above). The same is true at $p_1$ by symmetry. So we see $\E$ provides an example in Corollary \ref{Example}.

\end{document}